\documentclass{birkjour}

\usepackage[utf8]{inputenc}
\usepackage[T1]{fontenc}
\usepackage[english]{babel}
\usepackage{mathtools,amsmath,amssymb,amsfonts,amsthm,mathrsfs}
\usepackage{bbm}
\mathtoolsset{showonlyrefs}
\usepackage{enumitem}

\usepackage{geometry}
\usepackage{graphicx,color}
\usepackage{tikz,pgfplots,pgf}
\usepackage{subcaption}
\usepackage[font=small]{caption}

\definecolor{BF-blue}{RGB}{0,128,255}
\definecolor{BF-pink}{RGB}{255,0,128}
\definecolor{BF-red}{RGB}{255,0,64}
\definecolor{BF-green}{RGB}{128,255,0}

\newtheorem{theorem}{Theorem}[section]

\newtheorem{lemma}{Lemma}[section]

\newtheorem{corollary}{Corollary}[section]
\theoremstyle{definition}
\newtheorem{remark}{Remark}[section]

\numberwithin{equation}{section}

\begin{document}

\title[A Minkowski-curvature equation with super-exponential nonlinearity]{Positive solutions for a Minkowski-curvature \\equation with 
indefinite weight and \\super-exponential nonlinearity}

\author[A.~Boscaggin]{Alberto Boscaggin}

\address{
Department of Mathematics ``Giuseppe Peano'', University of Torino\\
Via Carlo Alberto 10, 10123 Torino, Italy}

\email{alberto.boscaggin@unito.it}

\author[G.~Feltrin]{Guglielmo Feltrin}

\address{
Department of Mathematics, Computer Science and Physics, University of Udine\\
Via delle Scienze 206, 33100 Udine, Italy}

\email{guglielmo.feltrin@uniud.it}

\author[F.~Zanolin]{Fabio Zanolin}

\address{
Department of Mathematics, Computer Science and Physics, University of Udine\\
Via delle Scienze 206, 33100 Udine, Italy}

\email{fabio.zanolin@uniud.it}

\thanks{Work written under the auspices of the Grup\-po Na\-zio\-na\-le per l'Anali\-si Ma\-te\-ma\-ti\-ca, la Pro\-ba\-bi\-li\-t\`{a} e le lo\-ro Appli\-ca\-zio\-ni (GNAMPA) of the Isti\-tu\-to Na\-zio\-na\-le di Al\-ta Ma\-te\-ma\-ti\-ca (INdAM). The first two authors are supported by INdAM--GNAMPA project ``Problemi ai limiti per l'equazione della curvatura media prescritta''.
\\
\textbf{Preprint -- July 2020}}



\date{}

\dedicatory{}

\begin{abstract}
We investigate the existence of positive solutions for a class of Minkowski-curvature equations with indefinite weight and nonlinear term having superlinear growth at zero and super-exponential growth at infinity. As an example, for the equation
\begin{equation*}
\Biggl{(} \dfrac{u'}{\sqrt{1-(u')^{2}}}\Biggr{)}' + a(t) \bigl{(}e^{u^{p}}-1\bigr{)}  = 0,
\end{equation*}
where $p > 1$ and $a(t)$ is a sign-changing function satisfying the mean-value condition
$\int_{0}^{T} a(t)\,\mathrm{d}t < 0$, we prove the existence of a positive solution for both periodic and Neumann boundary conditions.
The proof relies on a topological degree technique.
\end{abstract}

\maketitle

\section{Introduction}\label{section-1}

In this paper, we are concerned with the existence of positive $T$-periodic solutions to the differential equation
\begin{equation}\label{eq-intro}
\Biggl{(} \dfrac{u'}{\sqrt{1-(u')^{2}}}\Biggr{)}' + a(t) g(u) = 0,
\end{equation}
where $a \colon \mathbb{R} \to \mathbb{R}$ is a sign-changing $T$-periodic function and $g\colon \mathopen{[}0,+\infty\mathclose{[}\to \mathopen{[}0,+\infty\mathclose{[}$ is a continuous function vanishing only at $u = 0$.

As is well-known, equation \eqref{eq-intro} can be meant as a one-dimensional version of the mean curvature equation in the Lorentz--Minkowski space:
in recent years, the solvability of the associated boundary value problems has raised a lot of interests, both in the ODE and in the PDE setting
(see, for instance, \cite{Az-14,BeJeTo-13,BeMa-07,BoCoNo-pp,BoGa-19ccm,CoObOmRi-13,Da-16,Da-18,Hu-19,Ma-13,Po-18} and the references therein). On the other hand, the specific choice for the nonlinear term made in equation \eqref{eq-intro} casts it into the family of nonlinear problems with indefinite weight, a quite popular topic in Nonlinear Analysis since the pioneering papers \cite{AlTa-93,BaPoTe-88,BeCDNi-95,Bu-76,HeKa-80}.
In the case of linear differential operators, the bibliography about indefinite weight problems is nowadays very wide (we refer to \cite{Fe-18book,LG-16,PaZa-02} for an extensive list of references); more recently, the case of mean curvature equations, both in Euclidean and Minkowski space, 
has been considered as well (see, among others, \cite{BeZa-18,BoFe-PP,BoFe-20na,LGOm-20,LGOmRi-17jde,OmSo-PP}).

On this line of research, our present investigation is mainly motivated by some results obtained in \cite{BoFe-PP}. In particular, in \cite[Theorem~1.1]{BoFe-PP}, it was proved that the parameter-dependent equation
\begin{equation}\label{eq-potenza}
\Biggl{(} \dfrac{u'}{\sqrt{1-(u')^{2}}}\Biggr{)}' + \lambda a(t) u^{p} = 0,
\end{equation}
where $p > 1$ and $a(t)$ satisfies (besides some technical assumptions) the mean-value condition $\int_0^T a(t)\,\mathrm{d}t < 0$,
has at least two positive $T$-periodic solutions if the parameter $\lambda$ is positive and large enough (say, $\lambda > \lambda^*$)
and no positive $T$-periodic solutions if $\lambda$ is positive and sufficiently small (say, $\lambda \in (0,\lambda_*)$, with $\lambda_* \leq \lambda^*$). An analogous result for the Neumann boundary value problem was established in \cite{BoFe-20na}. 
Incidentally, let us recall that the condition $\int_0^T a(t)\,\mathrm{d}t < 0$
is actually necessary for the existence of a positive solution of \eqref{eq-potenza} with periodic or Neumann boundary conditions, 
as it can be easily checked by dividing the equation by $u^{p}$ and integrating by parts.

Actually, more general versions of these results are valid for the equation 
\begin{equation}\label{eq-g}
\Biggl{(} \dfrac{u'}{\sqrt{1-(u')^{2}}}\Biggr{)}' + \lambda a(t) g(u) = 0,
\end{equation}
with the same assumptions on $a(t)$ and suitable hypotheses on $g(u)$ at zero (always requiring a superlinear growth, that is $g(u)/u \to 0$ as $u \to 0^{+}$) and at infinity. In particular, according to \cite[Theorem~3.2]{BoFe-PP}, non-existence of solutions to \eqref{eq-g} for $\lambda$ small can be ensured when $g(u)$ is continuously differentiable and 
\begin{equation}\label{eq-nonex}
\limsup_{u \to +\infty} \frac{\vert g'(u) \vert}{(g(u))^\eta} < +\infty, \quad \text{for some $\eta \in \mathopen{[}0,1\mathclose{[}$}.
\end{equation}
Notice that the above condition implies that the growth of $g(u)$ at infinity is at most of power-type 
(precisely, $g(u) = O(u^\frac{1}{1-\eta})$ for $u \to +\infty$).

This solvability picture is in sharp contrast with the one of the semilinear indefinite equation
\begin{equation}\label{eq-semi}
u'' + \lambda a(t)g(u) = 0.
\end{equation}
Indeed, when $g(u) = u^p$ with $p > 1$ and, more in general, for a large class of functions $g(u)$ having superlinear growth both at zero and at infinity (namely, $g(u)/u \to 0$ as $u \to 0^{+}$ and $g(u)/u \to +\infty$ as $u \to +\infty$), only one positive $T$-periodic solution to \eqref{eq-semi} can be provided, but no assumptions on the parameter $\lambda$ are needed (see, for instance, \cite{FeZa-15ade},
as a counterpart of classical results in the PDE setting \cite{AlTa-93,BeCDNi-95}).

This essentially different behavior is, of course, a consequence of the nonlinear character of the Minkowski-curvature differential operator. 
Roughly speaking, it turns out that, in a Minkowski-curvature equation, the differential operator is predominant in dictating the behavior of large solutions. As a consequence, the role of the behavior of $g(u)$ at infinity becomes more subtle and, in particular, a power-type growth at infinity of $g(u)$ is no more sufficient to ensure solvability of equation \eqref{eq-g} without assumptions on the parameter $\lambda$.
This point of view is indeed confirmed by the fact that a two-solution theorem for $\lambda$ large, as well as non-existence for $\lambda$ small, can be proved for the semilinear equation \eqref{eq-semi} when $g(u)$ is a so-called super-sublinear function, meaning that $g(u)/u \to 0$ for both $u \to 0^{+}$ and $u \to +\infty$ (cf.~\cite{BoFeZa-16}).

The above discussion seems to suggest the rather unexpected conjecture that,
while keeping the same assumptions for $g(u)$ near zero, a stronger growth of $g(u)$ at infinity could instead ensure the solvability, for any $\lambda > 0$, of equation \eqref{eq-g} (written, from now on, simply as \eqref{eq-intro}, since the parameter $\lambda$ does not play a role).
Notice that, in any case, the growth of $g(u)$ at infinity should be so rapid as to violate condition \eqref{eq-nonex}: indeed, as already observed, in such a situation a non-existence result is valid.

The aim of the present paper is to provide a positive answer to this conjecture, by showing that an existence result can be established 
for a class of functions $g(u)$ with the property of growing at infinity more than their primitives, and thus exhibiting a ``super-exponential'' growth. The general statement will be given in Section~\ref{section-2}; by now, we just present it for the model example
\begin{equation}\label{eq-exp}
\Biggl{(} \frac{u'}{\sqrt{1-(u')^{2}}}\Biggr{)}' + a(t) \bigl{(}e^{u^{p}}-1\bigr{)} = 0,
\end{equation}
where $p > 1$ (notice that $e^{u^p}-1 \sim u^p$ for $u \to 0^+$). In such a situation, the following theorem holds true.

\begin{theorem}\label{th-intro}
Let $a \colon \mathbb{R} \to \mathbb{R}$ be a $T$-periodic sign-changing continuous function,
having a finite number of zeros in $\mathopen{[}0,T\mathclose{[}$ and satisfying the mean
value condition 
\begin{equation*}
\int_0^T a(t)\,\mathrm{d}t < 0.
\end{equation*}
Then, there exists a positive $T$-periodic solution of \eqref{eq-exp}.
\end{theorem}

A more general version of the above result, dealing with equation \eqref{eq-intro} paired with either periodic or Neumann boundary conditions, 
will be given in Theorem~\ref{th-main-ex}.

The proof of these results relies on a topological degree technique, following a line of research extensively developed in recent years (see \cite{Fe-18book} and the references therein).
More precisely, after having converted the boundary value problem into a fixed point equation in a Banach space $X$,
a suitably defined topological degree is evaluated on both small balls and large balls centered at the origin of $X$:
as a consequence of the assumptions on $g(u)$, these degrees turn out to be different, so that the existence of a solution follows by the excision property (its positivity is then recovered via a maximum principle argument). 
This strategy is, of course, very typical when dealing with superlinear problems associated with semilinear equations.
Here, however, the homotopy argument needed for the evaluation of the degree on large balls is quite subtle, 
due to the the interplay between the differential operator and the super-exponential behavior of $g(u)$ at infinity. 
This eventually provides some insights on an unexpected dynamics of large solutions to equation \eqref{eq-intro}
(see Remark~\ref{rem-3.1} for more comments about this).

\medskip

The paper is organized as follows. In Section~\ref{section-2}, we state our main result (Theorem~\ref{th-main-ex}) for a general class  of weights $a(t)$ and nonlinearities $g(u)$, the key super-exponential assumptions being given by condition $(g_{\mathrm{SE}})$. Some comments on the hypotheses and some numerical simulations are also given. Section~\ref{section-3} contains the details of the proof.

\section{Statement of the main result}\label{section-2}

In this section, we give the statement of our main result, dealing with the boundary value problem
\begin{equation}\label{eq-main-phi}
\begin{cases}
\, \Biggl{(} \dfrac{u'}{\sqrt{1-(u')^{2}}}\Biggr{)}' + a(t) g(u) = 0, \vspace{0.1cm}\\
\, \mathfrak{B}(u)=0,
\end{cases}
\end{equation}
where the boundary operator $\mathfrak{B}\colon\mathcal{C}^{1}(\mathopen{[}0,T\mathclose{]})\to\mathbb{R}^{2}$ is either of periodic or Neumann type, that is
\begin{equation}\label{b-per}
\mathfrak{B}(u)=\bigl{(}u(T)-u(0),u'(T)-u'(0)\bigr{)}
\end{equation}
or
\begin{equation}\label{b-neu}
\mathfrak{B}(u)=\bigl{(}u'(0),u'(T)\bigr{)}.
\end{equation}
We also assume that $g\colon \mathopen{[}0,+\infty\mathclose{[}\to \mathopen{[}0,+\infty\mathclose{[}$ is a continuous function
satisfying 
\begin{itemize}[leftmargin=30pt,labelsep=12pt,topsep=5pt]
\item [$(g_{*})$] \textit{$g(0) = 0$ and $g(u) > 0$ for $u>0$},
\end{itemize}
and $a \colon \mathopen{[}0,T\mathclose{]} \to \mathbb{R}$ is a measurable and essentially bounded function satisfying the following technical condition 
\begin{itemize}[leftmargin=30pt,labelsep=12pt,topsep=5pt]
\item [$(a_{*})$]
\textit{there exists a finite number of points
\begin{equation*}
0 = \tau_{0} \leq \sigma_{1} < \tau_{1} < \sigma_{2} < \tau_{2} < \ldots < \sigma_{m} < \tau_{m} \leq \sigma_{m+1} = T
\end{equation*}
such that
\begin{align*}
&\qquad\qquad a(t)>0, \; \text{ for a.e.~$t\in I^{+}_{i}= \mathopen{[}\sigma_{i},\tau_{i}\mathclose{]}$, for $i=1,\ldots,m$,} \\
&\qquad\qquad a(t)\leq0, \; \text{ for a.e.~$t\in I^{-}_{i}= \mathopen{[}\tau_{i},\sigma_{i+1}\mathclose{]}$, for $i=0,\ldots,m$.}
\end{align*}}
\end{itemize}
Observe that the above condition is certainly satisfied if $a(t)$ has a finite number of zeros in $\mathopen{[}0,T\mathclose{[}$ (as assumed in Theorem~\ref{th-intro}). However, it also allows the presence of intervals on which $a(t)$ vanishes: indeed, it could happen that 
$a(t) \equiv 0$ on $I^{-}_{i}$ for some $i$. Basically, the only situation to be prevented is an infinite number of changes of sign for $a(t)$.

\begin{remark}\label{rem-nodal}
As well known, the boundary value problem \eqref{eq-main-phi} with the periodic boundary conditions $u(0)=u(T)$, $u'(0)=u'(T)$ is equivalent to the search for $T$-periodic solutions of the equation, with $a(t)$ extended on the real line by $T$-periodicity. In view of this remark, in this situation it is not restrictive to suppose that 
\begin{equation*}
\sigma_{1}=0 \quad \text{and} \quad \tau_{m} < T,
\end{equation*}
namely that in $\mathopen{[}0,T\mathclose{]}$ the weight is positive in a right neighborhood of $t=0$ and less than or equal to zero in a left neighborhood of $t=T$: indeed, this can always be achieved by a time-shift of the equation. This observation will be useful in Section~\ref{section-3}, when proving the main result.
\hfill$\lhd$
\end{remark}

In order to state our theorem, we further need to introduce some constants depending on the behavior of $a(t)$ on the positivity intervals.
Precisely, we define, for $i=1,\ldots,m$, the constant $A_{i}$ as the minimum of the $L^{1}$-norm of $a(t)$ over all the intervals of length 
$(\tau_{i}-\sigma_{i})/4$ contained in $I^{+}_{i}$, that is
\begin{equation}\label{def-Ai}
A_{i} = \min \biggl{\{} \|a\|_{L^{1}(t_1,t_2)} \colon \mathopen{[}t_1,t_2\mathclose{]}\subseteq I^{+}_{i}, t_2-t_1=\frac{\tau_{i}-\sigma_{i}}{4} \biggr{\}}.
\end{equation}
Notice that the existence of the minimum is ensured by the continuity of the function $A(t) = \int_{0}^{t} a(\xi) \,\mathrm{d}\xi$ and the fact that
\begin{equation*}
A_{i} = \min_{t \in \bigl{[}\sigma_{i} + \frac{\tau_{i}-\sigma_{i}}{4},\tau_{i}\bigr{]}} \biggl{(} A(t) - A\biggl{(}t-\frac{\tau_{i}-\sigma_{i}}{4}\biggr{)} \biggr{)}.
\end{equation*}
Furthermore, $A_{i} > 0$, since $a(t) > 0$ a.e.~in $I^{+}_{i}$ by condition $(a_{*})$. 

The main result of the present paper reads as follows (in condition $(g_{\mathrm{SE}})$, the symbol $a^-$ stands for the negative part of the weight function, that is, $a^-(t) = -\min\{a(t),0\}$).

\begin{theorem}\label{th-main-ex}
Let $a \colon \mathopen{[}0,T\mathclose{]}\to\mathbb{R}$ be a measurable and essentially bounded function satisfying $(a_{*})$ and 
\begin{itemize}[leftmargin=34pt,labelsep=12pt,itemsep=6pt,topsep=5pt]
\item [$(a_{\#})$] $\displaystyle \int_{0}^{T} a(t) \,\mathrm{d}t <0$.
\end{itemize}
Let $g\colon \mathopen{[}0,+\infty\mathclose{[}\to \mathopen{[}0,+\infty\mathclose{[}$ be a continuous function satisfying
$(g_{*})$ and
\begin{itemize}[leftmargin=34pt,labelsep=12pt,itemsep=6pt,topsep=5pt]
\item [$(g_{0})$] $\displaystyle \lim_{u\to 0^{+}} \dfrac{g(u)}{u} = 0$ and $\displaystyle \lim_{\substack{u\to0^{+} \\ \omega\to1}}\dfrac{g(\omega u)}{g(u)}=1$;
\item [$(g_{\nearrow})$] $g$ is monotone non-decreasing on $\mathopen{[}\hat{R},+\infty\mathclose{[}$, for some $\hat{R} > 0$;
\item [$(g_{\mathrm{SE}})$] 
$\displaystyle 
\liminf_{u\to +\infty} \dfrac{g(u)}{G(u)} > \dfrac{\|a^{-}\|_{L^{\infty}(0,T)}}{\displaystyle \min_{i=1,\ldots,m}A_{i}}$, \,
where $\displaystyle G(u)=\int_{0}^{u} g(s)\,\mathrm{d}s$. 
\end{itemize}
Then, the boundary value problem \eqref{eq-main-phi} admits at least one positive solution.
\end{theorem}

As usual, a solution to \eqref{eq-main-phi} is meant as a continuously differentiable function $u \colon \mathopen{[}0,T\mathclose{]} \to \mathbb{R}$, such that $|u'(t)| < 1$ for every $t \in \mathopen{[}0,T\mathclose{]}$, the map
$t \mapsto {u'(t)}/{\sqrt{1-(u'(t))^{2}}}$ is absolutely continuous on $\mathopen{[}0,T\mathclose{]}$, the differential equation is satisfied for a.e.~$t \in \mathopen{[}0,T\mathclose{]}$ and $\mathfrak{B}(u)=0$. We say that a solution $u(t)$ is positive if $u(t)>0$ for every $t \in \mathopen{[}0,T\mathclose{]}$.

Some more comments about the hypotheses on $g(u)$ in Theorem~\ref{th-main-ex} are now in order.

\begin{remark}[Condition for $g(u)$ at zero]\label{rem-2.2}
Condition $(g_{0})$ is the same as the one used in \cite{BoFe-PP}: it requires a superlinear growth of $g(u)$ near zero 
(that is, $g(u)/u \to 0$ as $u \to 0^{+}$) as well as a so-called \textit{regular oscillation behavior} near zero (cf.~\cite{DjTo-01,FeZa-15ade} for more comments about this kind of assumption). A simple situation in which $(g_{0})$ holds true is when $g(u)$ has a superlinear power-type growth at zero, namely $g(u) \sim C u^p$ for $u \to 0^+$, for some $C > 0$ and $p > 1$.
\hfill$\lhd$
\end{remark}

\begin{remark}[Condition for $g(u)$ at infinity]\label{rem-2.3}
Besides the monotonicity assumption $(g_{\nearrow})$, condition $(g_{\mathrm{SE}})$ is the crucial hypothesis of Theorem~\ref{th-main-ex}. 
It implies a rapid growth for $g(u)$ at infinity: indeed, since the ratio $g(u)/G(u)$ is the logarithmic derivative of $G(u)$, an integration yields
\begin{equation}\label{G-infty}
G(u) \geq \alpha \, e^{K u}, \quad \text{for all $u \geq R$,}
\end{equation}
for suitable constants $\alpha, R > 0$ and $K = \|a^{-}\|_{L^{\infty}(0,T)}/\min_{i} A_{i}$, and finally
\begin{equation}\label{g-infty}
g(u) \geq \alpha K e^{K u}, \quad \text{for all $u \geq R$.}
\end{equation}
Hence, the growth at infinity of $g(u)$ is at least of exponential type. 

The more relevant situations in which condition $(g_{\mathrm{SE}})$ holds true are the ones for which
\begin{equation*}
\lim_{u\to +\infty} \dfrac{g(u)}{G(u)} = +\infty.
\end{equation*}
In this case, a minor variant of the above argument shows that, for every $L > 0$, it holds that
\begin{equation*}
g(u) \geq \beta_{L} e^{Lu}, \quad \text{for all $u \geq R_{L}$,}
\end{equation*}
for suitable constants $\beta_{L}, R_{L} > 0$. In this case, the growth at infinity of $g(u)$ is thus ``super-exponential''.

Incidentally, let us notice that assumption $(g_{\mathrm{SE}})$ prevents the validity of condition
\eqref{eq-nonex}, which, as discussed in the introduction, implies a growth of $g(u)$ at most of power type. This is consistent with the fact that, when \eqref{eq-nonex} holds, the boundary value problem~\eqref{eq-main-phi} can be non-solvable.
\hfill$\lhd$
\end{remark}

Let us observe that, in view of the generalized version of L'H\^{o}pital rule \cite{Ta-52}, assumption $(g_{\mathrm{SE}})$ holds true whenever the following condition is satisfied:
\begin{itemize}[leftmargin=34pt,labelsep=12pt,itemsep=6pt,topsep=5pt]
\item [$(g_{\mathrm{SE}}')$] \textit{$g(u)$ is differentiable in a neighborhood of infinity and 
\begin{equation*}
\liminf_{u\to +\infty} \dfrac{g'(u)}{g(u)} > \dfrac{\|a^{-}\|_{L^{\infty}(0,T)}}{\displaystyle \min_{i=1,\ldots,m}A_{i}}.
\end{equation*}
}
\end{itemize}
In this case, the monotonicity assumption $(g_{\nearrow})$ is automatically guaranteed (since $g'(u) > 0$ for $u$ large). Hence, we can provide the following corollary of Theorem~\ref{th-main-ex}.

\begin{corollary}\label{cor-2.1}
Let $a \colon \mathopen{[}0,T\mathclose{]}\to\mathbb{R}$ be a measurable and essentially bounded function satisfying $(a_{*})$ and 
$(a_{\#})$.
Let $g\colon \mathopen{[}0,+\infty\mathclose{[}\to \mathopen{[}0,+\infty\mathclose{[}$ be a continuous function satisfying
$(g_{*})$, $(g_{0})$ and $(g_{\mathrm{SE}}')$.
Then, the boundary value problem \eqref{eq-main-phi} admits at least one positive solution.
\end{corollary}

Theorem~\ref{th-intro} follows directly from Corollary~\ref{cor-2.1}: indeed, it is easily checked that the function $g(u) = e^{u^p}-1$ with $p > 1$ satisfies conditions $(g_{*})$, $(g_{0})$ and $(g_{\mathrm{SE}}')$. See also Figure~\ref{fig-01} for a numerical simulation.
Another example of nonlinear term satisfying the assumptions of Corollary~\ref{cor-2.1} is $g(u) = u^p e^{\kappa u}$ with $p > 1$ and $k$ positive and sufficiently large (see Figure~\ref{fig-02}). 

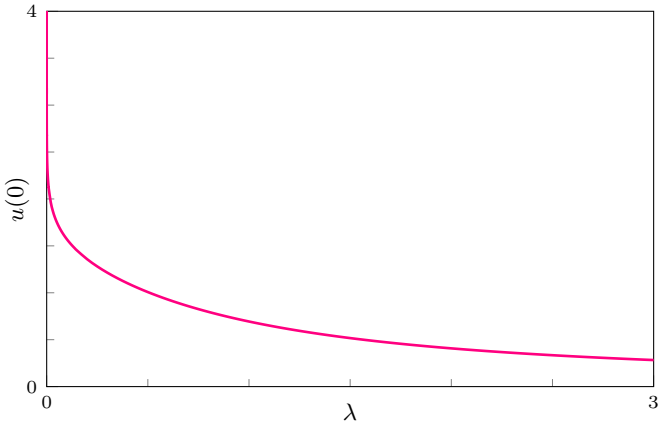
\begin{figure}[!htb]
\centering
\begin{tikzpicture}[scale=1]
\begin{axis}[
  tick pos=left,
  tick label style={font=\scriptsize},
          scale only axis,
  enlargelimits=false,
  xtick={0,3},
  ytick={0,4},
  xlabel={\small $\lambda$},
  ylabel={\small $u(0)$},
  max space between ticks=30,
                minor x tick num=5,
                minor y tick num=7,  
every axis x label/.style={
below,
at={(4cm,0cm)},
  yshift=-3pt
  },
every axis y label/.style={
below,
at={(0cm,2.5cm)},
  xshift=-3pt},
  y label style={rotate=90,anchor=south},
  width=8cm,
  height=5cm,  
  xmin=0,
  xmax=3,
  ymin=0,
  ymax=4]
\addplot [color=BF-pink,line width=1pt,smooth] coordinates {(0,5) (0.00001,3.285425155) (0.00002,3.184202141) (0.00003,3.123684202) (0.00004,3.080131474) (0.00005,3.045986018) (0.00006,3.017845333) (0.00007,2.993879427) (0.00008,2.972988538) (0.00009,2.954459182) (0.0001,2.937801729) (0.00011,2.922665370) (0.00012,2.908790024) (0.00013,2.895977452) (0.00014,2.884073030) (0.00015,2.872953798) (0.00016,2.862520337) (0.00017,2.852691105) (0.00018,2.843398381) (0.00019,2.834585297) (0.0002,2.826203631) (0.00021,2.818212134) (0.00022,2.810575242) (0.00023,2.803262074) (0.00024,2.796245641) (0.00025,2.789502221) (0.00026,2.783010842) (0.00027,2.776752880) (0.00028,2.770711715) (0.00029,2.764872452) (0.0003,2.759221693) (0.00031,2.753747340) (0.00032,2.748438572) (0.00033,2.743285144) (0.00034,2.738278115) (0.00035,2.733409180) (0.00036,2.728670728) (0.00037,2.724055765) (0.00038,2.719557848) (0.00039,2.715171030) (0.0004,2.710889808) (0.00041,2.706709083) (0.00042,2.702624118) (0.00043,2.698630507) (0.00044,2.694724142) (0.00045,2.690901189) (0.00046,2.687158064) (0.00047,2.683491408) (0.00048,2.679898076) (0.00049,2.676375111) (0.0005,2.672919732) (0.001,2.552136569) (0.0015,2.479353222) (0.002,2.426682292) (0.0025,2.385203775) (0.003,2.350889129) (0.0035,2.321566295) (0.004,2.295927275) (0.0045,2.273122009) (0.005,2.252566252) (0.0055,2.233840834) (0.006,2.216634617) (0.0065,2.200710176) (0.007,2.185882136) (0.0075,2.172002941) (0.008,2.158953181) (0.0085,2.146634838) (0.009,2.134966440) (0.0095,2.123879527) (0.01,2.113316005) (0.02,1.965730903) (0.03,1.874425637) (0.04,1.806835435) (0.05,1.752499277) (0.06,1.706679090) (0.07,1.666814589) (0.08,1.631361193) (0.09,1.599313551) (0.1,1.569979788) (0.11,1.542862961) (0.12,1.517593779) (0.13,1.493890254) (0.14,1.471532041) (0.15,1.450343634) (0.2,1.357754115) (0.25,1.280684099) (0.3,1.213914948) (0.35,1.154629560) (0.4,1.101128721) (0.45,1.052303899) (0.5,1.007387157) (0.55,0.9658205698) (0.6,0.9271815080) (0.65,0.8911388518) (0.7,0.8574252858) (0.75,0.8258194272) (0.8,0.7961339301) (0.85,0.7682073682) (0.9,0.7418985618) (0.95,0.7170825438) (1.,0.6936476186) (1.05,0.6714931803) (1.1,0.6505280562) (1.15,0.6306692309) (1.2,0.6118408279) (1.25,0.5939732869) (1.3,0.5770026867) (1.35,0.5608701742) (1.4,0.5455214806) (1.45,0.5309064975) (1.5,0.5169789045) (1.55,0.5036958407) (1.6,0.4910176110) (1.65,0.4789074221) (1.7,0.4673311443) (1.75,0.4562570937) (1.8,0.4456558362) (1.85,0.4355000084) (1.9,0.4257641536) (1.95,0.4164245727) (2.,0.4074591866) (2.05,0.3988474116) (2.1,0.3905700436) (2.15,0.3826091538) (2.2,0.3749479914) (2.25,0.3675708957) (2.3,0.3604632136) (2.35,0.3536112260) (2.4,0.3470020786) (2.45,0.3406237191) (2.5,0.3344648288) (2.55,0.3285148111) (2.6,0.3227636790) (2.65,0.3172020532) (2.7,0.3118211064) (2.75,0.3066125302) (2.8,0.3015684954) (2.85,0.2966816220) (2.9,0.2919449464) (2.95,0.2873518954) (3.,0.2828962595)};
\end{axis}
\end{tikzpicture}
\caption{Bifurcation diagram with bifurcation parameter $\lambda\in\mathopen{[}0,3\mathclose{]}$ for
problem \eqref{eq-main-phi} with Neumann boundary conditions, $a(t) = 1$ on $\mathopen{[}0,1\mathclose{[}$ and $a(t) = -10$ on $\mathopen{[}1,2\mathclose{]}$, $g(u)=g_\lambda(u) = \lambda (e^{u^{2}}-1)$. We have inserted here the parameter $\lambda$ in order to stress that the existence of a positive solution can be ensured for every $\lambda > 0$. It can be proved that the solution $u_\lambda(t)$ explodes when $\lambda \to 0^+$ (cf. Remark~\ref{rem-2.4}); 
however, it appears from the numerical simulation that the velocity of explosion is very slow.}    
\label{fig-01}
\end{figure}

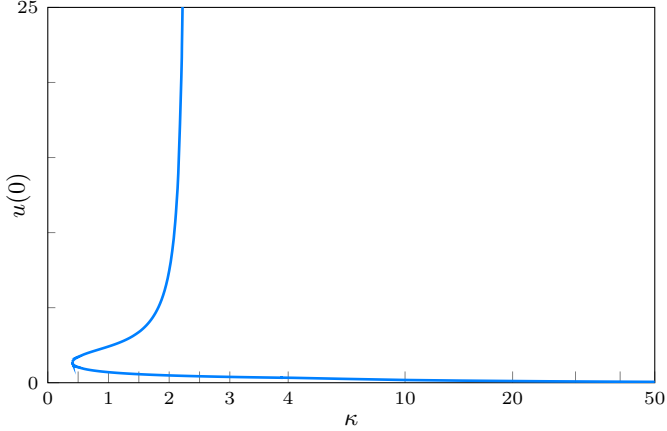
\begin{figure}[!htb]
\centering
\begin{tikzpicture}[scale=1]
\begin{axis}[
  tick pos=left,
  tick label style={font=\scriptsize},
          scale only axis,
  enlargelimits=false,
  xtick={0,0.5,1,1.5,2,2.5,2.99873,3.96072,5.88592,7.65776,8.69422,9.42959,10},
  xticklabels={$0$,,$1$,,$2$,,$3$,$4$,$10$,$20$,,,$50$},
  ytick={0,25},
  xlabel={\small $\kappa$},
  ylabel={\small $u(0)$},
  max space between ticks=30,
                minor x tick num=0,
                minor y tick num=4,  
every axis x label/.style={
below,
at={(4cm,-0.2cm)},
  yshift=-3pt
  },
every axis y label/.style={
below,
at={(0cm,2.5cm)},
  xshift=-3pt},
  y label style={rotate=90,anchor=south},
  width=8cm,
  height=5cm,  
  xmin=0,
  xmax=10,
  ymin=0,
  ymax=25]
\addplot [color=BF-blue,line width=1pt,smooth] coordinates  {(0.4042,1.359794328) (0.4041,1.353988519) (0.4041,1.337915751) (0.4042,1.332171838) (0.4043,1.328215379) (0.4044,1.325001802) (0.4045,1.322226061) (0.4046,1.319747882) (0.4047,1.317488951) (0.4048,1.315400260) (0.4049,1.313448727) (0.405,1.311610772) (0.4075,1.281112381) (0.41,1.261304120) (0.42,1.209585398) (0.43,1.174067652) (0.44,1.145655527) (0.45,1.121504576) (0.46,1.100277545) (0.47,1.081217994) (0.48,1.063848495) (0.49,1.047844228) (0.5,1.032972978) (0.6,0.9216620781) (0.7,0.8458923845) (0.8,0.7879167408) (0.9,0.7409658243) (1.,0.7016064121) (1.1,0.6678221659) (1.2,0.6383161142) (1.3,0.6121996009) (1.4,0.5888350746) (1.5,0.5677486226) (1.6,0.5485779574) (1.7,0.5310396920) (1.8,0.5149079030) (1.9,0.4999995575) (2.,0.4861643014) (2.1,0.4732771477) (2.2,0.4612330913) (2.3,0.4499430924) (2.4,0.4393310208) (2.5,0.4293312976) (2.59996,0.4198869701) (2.69983,0.4109485817) (2.79961,0.4024725123) (2.89925,0.3944203671) (2.99873,0.3867581257) (3.09798,0.3794555146) (3.19693,0.3724854846) (3.29551,0.3658237774) (3.3936,0.3594485613) (3.49107,0.3533401250) (3.58773,0.3474806168) (3.68337,0.3418538228) (3.77769,0.3364449771) (3.87032,0.3312405983) (4.04443,0.3262283480) (5.88592,0.1816510202) (7.65776,0.1121706449) (8.69422,0.08351003500) (9.42959,0.06739322586) (10,0.05691657637)};

\addplot [color=BF-blue,line width=1pt,smooth] coordinates  {(0.4042,1.332171838) (0.4041,1.337915751) (0.4041,1.353988519) (0.4042,1.359794328) (0.4043,1.363816722) (0.4044,1.367096223) (0.4045,1.369936674) (0.4045,1.369936674) (0.4046,1.372480467) (0.4047,1.374804549) (0.4048,1.376958364) (0.4049,1.378975237) (0.4049,1.378975237) (0.405,1.380878275) (0.4075,1.413001954) (0.41,1.434432176) (0.42,1.492615193) (0.43,1.534562218) (0.44,1.569370884) (0.45,1.599889523) (0.46,1.627457829) (0.47,1.652836142) (0.48,1.676504662) (0.49,1.698791219) (0.5,1.719930778) (0.6,1.893814461) (0.7,2.033981836) (0.8,2.161076303) (0.9,2.284932717) (1.,2.412493115) (1.1,2.550097950) (1.2,2.704705994) (1.3,2.884926493) (1.4,3.102228543) (1.5,3.372795150) (1.55,3.535191155) (1.6,3.720986893) (1.65,3.935649452) (1.7,4.186430373) (1.75,4.483178581) (1.8,4.839645534) (1.85,5.275645995) (1.9,5.820829826) (1.95,6.521659190) (2.,7.455302297) (2.05,8.760026142) (2.1,10.71060032) (2.15,13.94194837) (2.2,20.32681447) (2.21,22.45505803) (2.22,25.06847328) (2.23,28.40144733)};
\end{axis}
\end{tikzpicture}
\caption{Bifurcation diagram with bifurcation parameter $\kappa\in\mathopen{[}0,50\mathclose{]}$ for
problem \eqref{eq-main-phi} with Neumann boundary conditions, $a(t) = 1$ on $\mathopen{[}0,1\mathclose{[}$ and $a(t) = -10$ on $\mathopen{[}1,2\mathclose{]}$, 
$g(u)=g_\kappa(u) = u^{2}e^{\kappa u}$. The horizontal  axis is linear up to $2.5$ and then smoothly switches to a logarithmic scale. In this case, $g'(u)= e^{\kappa u} (\kappa u^2 + 2u)$ so that $g'(u)/g(u)\to \kappa$ as $u\to+\infty$. Moreover, $A_1 = 1/4$ and $\|a^{-}\|_{L^{\infty}(0,2)} = 10$. Hence, condition $(g_{\mathrm{SE}}')$ 
reads as $\kappa > 40$ (it is possible to see that $g(u)/G(u) \to \kappa$ as $u \to +\infty$, as well; hence, there is no advantage in considering condition $(g_{\mathrm{SE}})$ instead of $(g_{\mathrm{SE}}')$). The numerical simulation, however, seems to suggest that the existence of a positive solution is ensured for a larger range of the parameter $\kappa$; moreover, multiplicity phenomena seem to appear. On the other hand, it is worth noticing that the bifurcation branch does not project on all the positive values of $\kappa$. }     
\label{fig-02}
\end{figure}

\begin{remark}\label{rem-2.4}
In the case of the parameter-dependent boundary value problem
\begin{equation}\label{eq-main2-phi}
\begin{cases}
\, \Biggl{(} \dfrac{u'}{\sqrt{1-(u')^{2}}}\Biggr{)}' + \lambda a(t) g(u) = 0, \vspace{0.1cm}\\
\, \mathfrak{B}(u)=0,
\end{cases}
\end{equation}
with $a(t)$ and $g(u)$ satisfying the assumptions of Theorem~\ref{th-main-ex}, the existence of a positive solution $u_\lambda(t)$ follows for any
$\lambda > 0$. We claim that 
\begin{equation*}
\lim_{\lambda \to 0^+} \| u_\lambda \|_{\infty} = +\infty,
\end{equation*}
where $\| u_\lambda \|_{\infty} = \max_{t \in \mathopen{[}0,T\mathclose{]}} | u_\lambda(t) |$.
Indeed, if we assume by contradiction that $\| u_\lambda \|_{\infty} \leq M$ for some $M > 0$ and $\lambda$ small, then $u_\lambda(t)$ solves problem \eqref{eq-main2-phi} with the modified nonlinearity $g_M(u) = g(\min\{u,M\})$.
Hence, we enter the setting of \cite[Theorem~3.2]{BoFe-PP}, implying non-existence of solutions for $\lambda$ small enough.

Incidentally, let us also observe that the existence of a family of positive solutions $u_\lambda(t)$ for $\lambda$ large
was already proved in \cite[Theorem~3.3]{BoFe-PP}, dealing with the periodic problem (the Neumann case can be treated as well,  
using the approach in \cite{BoFe-20na}). By \cite[Theorem~4.1]{BoFe-PP}, it holds that $\| u_\lambda \|_{\infty} \to
0$ as $\lambda \to +\infty$.
\hfill$\lhd$
\end{remark}

\section{Proof of the main result}\label{section-3}

This section is devoted to the proof of Theorem~\ref{th-main-ex}. We first present a semi-abstract result for a general class of second-order problems, which is obtained via topological degree techniques (see Section~\ref{section-3.1}); then we apply it in our framework to prove Theorem~\ref{th-main-ex} (see Section~\ref{section-3.2}).

\subsection{A semi-abstract result}\label{section-3.1}

Let us consider the boundary value problem
\begin{equation}\label{pb-main}
\begin{cases}
\, (\varphi(u'))' + a(t) g(u) = 0,\\
\, \mathfrak{B}(u)=0,
\end{cases}
\end{equation}
where $\varphi\colon \mathopen{]}-1,1\mathclose{[} \to \mathbb{R}$ is a homeomorphism with $\varphi(0)=0$ and the boundary operator $\mathfrak{B}(u)$ is of periodic or Neumann type (see \eqref{b-per} and \eqref{b-neu}). 
Notice that problem \eqref{eq-main-phi} enters this setting with the choice
\begin{equation}\label{def-phi}
\varphi(s) = \dfrac{s}{\sqrt{1-s^{2}}}.
\end{equation}
In this more general context, the following result holds true.

\begin{lemma}\label{lem-deg}
Let $a \colon \mathopen{[}0,T\mathclose{]} \to \mathbb{R}$ be a Lebesgue integrable function satisfying $(a_{\#})$ and let $g \colon \mathopen{[}0,+\infty\mathclose{[} \to \mathopen{[}0,+\infty\mathclose{[}$ be a continuous function satisfying $(g_{*})$. Assume that there exist $r,R\in\mathbb{R}$ with $0<r<R$ and $v\in L^{1}(0,T)$ with $v\geq0$ and $v\not\equiv0$ such that the following properties hold.
\begin{itemize}[leftmargin=34pt,labelsep=12pt,itemsep=6pt,topsep=5pt]
\item[$(H_{1})$]
If $\vartheta\in \mathopen{]}0,1\mathclose{]}$ and $u(t)$ is a solution of
\begin{equation}\label{eq-lem-deg1}
\begin{cases}
\, (\varphi(u'))' + \vartheta a(t) g(u) = 0,\\
\, \mathfrak{B}(u)=0,
\end{cases}
\end{equation}
then $\|u\|_{\infty} \neq r$.
\item[$(H_{2})$]
If $\alpha \geq 0$ and $u(t)$ is a solution of
\begin{equation}\label{eq-lem-deg0}
\begin{cases}
\, (\varphi(u'))' + a(t) g(u) + \alpha v(t) = 0,\\
\, \mathfrak{B}(u)=0,
\end{cases}
\end{equation}
then $\|u\|_{\infty}\neq R$.
\item[$(H_{3})$]
There exists $\alpha_{0} \geq 0$ such that problem \eqref{eq-lem-deg0}, with $\alpha=\alpha_{0}$, has no solutions $u(t)$ with $\|u\|_{\infty}\leq R$.
\end{itemize}
Then, there exists a solution $u(t)$ of \eqref{pb-main} with $r<\|u\|_{\infty} = \displaystyle \max_{t\in\mathopen{[}0,T\mathclose{]}} u(t) <R$.
\end{lemma}

\begin{proof}
The proof is based on a topological degree argument and follows a scheme similar to \cite[Section~2 and Appendix~B]{BoFe-PP} for the periodic problem and \cite[Section~2]{BoFe-20na} for the Neumann problem. Due to the present unified setting and for the sake of completeness, we give here a sketch of the proof.

As a first step, we introduce the $L^{1}$-Carath\'{e}odory function
\begin{equation*}
f(t,u) = 
\begin{cases}
\, a(t)g(u), & \text{if $u\geq 0$,} \\
\, -u, & \text{if $u<0$,}
\end{cases}
\end{equation*}
and the two-parameter-dependent boundary value problem
\begin{equation}\label{eq-hom}
\begin{cases}
\, (\varphi(u'))' + \vartheta \bigl{[} f(t,u) + \alpha v \bigr{]} = 0,\\
\, \mathfrak{B}(u)=0,
\end{cases}
\end{equation}
where $\vartheta\in\mathopen{[}0,1\mathclose{]}$ and $\alpha\geq0$. From the definition of $f(t,u)$, via a maximum principle argument (cf.~\cite[Appendix~A]{BoFe-PP}), we deduce that every solution of \eqref{eq-hom} is non-negative.

As a second step, we reduce problem \eqref{eq-hom} to an equivalent fixed point problem in a Banach space $X$, namely to a problem of the form
\begin{equation}\label{eq-fix}
w = \Psi_{\vartheta,\alpha} w, \quad \text{$w\in X$.}
\end{equation}
The definitions of $X$ and of the operator $\Psi_{\vartheta,\alpha}\colon X \to X$ depend on the boundary conditions. More precisely, in the periodic setting, writing the differential equation in \eqref{eq-hom} as a system of two first-order equations and following the arguments in \cite[Section~2]{BoFe-PP}, we can define $\Psi_{\vartheta,\alpha}$ in $X=\mathcal{C}(\mathopen{[}0,T\mathclose{]},\mathbb{R}^{2})$ with
$w = (u,\varphi(u'))$. On the other hand, for Neumann boundary conditions, we can consider the integral operator defined in \cite[Section~2.1]{BoFe-20na} with $X=\mathcal{C}(\mathopen{[}0,T\mathclose{]})$ and $w = u$. We remark that the definitions of the fixed point operator $\Psi_{\vartheta,\alpha}$ are variants of the one introduced in \cite{MaMa-98,Ma-13}.

As a third step, we observe that $\Psi_{\vartheta,\alpha}$ is a completely continuous operator.
Hence, if $\Omega\subseteq  X$ is an open set 
such that $\{w \in \Omega \colon w = \Psi_{\vartheta,\alpha} w \}$ is compact, we can consider the Leray--Schauder topological degree $\mathrm{deg}_{\mathrm{LS}}(\mathrm{Id}_{X}-\Psi_{\vartheta,\alpha},\Omega)$ (notice that, since we are not assuming the boundedness of $\Omega$, we have to rely on the extension of the degree for locally compact operators, see for instance \cite{Nu-85,Nu-93}). 

As described in the second step, in order to find a solution of \eqref{pb-main}, we need to look for fixed points of $\Psi_{1,0}$. Accordingly, thanks to the existence property of the degree, we are going to show that $\mathrm{deg}_{\mathrm{LS}}(\mathrm{Id}_{X}-\Psi_{1,0},\Omega)\neq0$ for an open set $\Omega\subseteq X$ not containing the trivial solution $w\equiv0$.

We consider the sets $\Omega_{d}\subseteq X$ with $d\in\{r,R\}$, where $0<r<R$ are given in hypotheses $(H_{1})$, $(H_{2})$ and $(H_{3})$. Depending on the boundary conditions, we have $\Omega_{d}=B(0,d)\times \mathcal{C}(\mathopen{[}0,T\mathclose{]})\subseteq\mathcal{C}(\mathopen{[}0,T\mathclose{]},\mathbb{R}^{2})$ for periodic boundary conditions and $\Omega_{d}=B(0,d)\subseteq\mathcal{C}(\mathopen{[}0,T\mathclose{]})$ for Neumann boundary conditions, with $B(0,d)$ the open ball centered at the origin and with radius $d$ in the space $\mathcal{C}(\mathopen{[}0,T\mathclose{]})$.

We claim that 
\begin{equation}\label{eq-deg1} 
\lvert \mathrm{deg}_{\mathrm{LS}}(\mathrm{Id}_{X}-\Psi_{1,0},\Omega_{r}) \rvert = 1.
\end{equation}
Let us consider problem \eqref{eq-hom} with $\alpha=0$, where $\vartheta\in\mathopen{[}0,1\mathclose{]}$ is thought as a homotopic parameter. Since solutions of \eqref{eq-hom} are non-negative and thus solve \eqref{eq-lem-deg1}, hypothesis $(H_{1})$ ensures that the degree $\mathrm{deg}_{\mathrm{LS}}(\mathrm{Id}_{X}-\Psi_{\vartheta,0},\Omega_{r})$ is well-defined for every $\vartheta\in\mathopen{]}0,1\mathclose{]}$. If $\vartheta=0$, the reduction property of the degree (cf.~\cite[Appendix~B]{BoFe-PP} and \cite[Appendix~A]{Fe-18book}) gives
\begin{equation*}
\lvert \mathrm{deg}_{\mathrm{LS}}(\mathrm{Id}_{X}-\Psi_{0,0},\Omega_{r}) \lvert = \lvert \mathrm{deg}_{\mathrm{B}}(-f^{\#},\mathopen{]}-r,r\mathclose{[})\lvert,
\end{equation*}
where ``$\mathrm{deg}_{\mathrm{B}}$'' denotes the finite-dimensional Brouwer degree and
\begin{equation*}
f^{\#}(s) =
\begin{cases}
\, \displaystyle g(s) \int_{0}^{T} a(t)\,\mathrm{d}t , & \text{if $s\geq0$,} \\
\, -s, & \text{if $s<0$.}
\end{cases}
\end{equation*} 
By $(a_{\#})$, since $f^{\#}(s) s <0$ for all $s\neq0$, we deduce that
\begin{equation*}
\lvert \mathrm{deg}_{\mathrm{B}}(-f^{\#},\mathopen{]}-r,r\mathclose{[}) \lvert = 1.
\end{equation*}
By the homotopy invariance property of the degree, we have the claim.

We claim that 
\begin{equation}\label{eq-deg0} 
\mathrm{deg}_{\mathrm{LS}}(\mathrm{Id}_{X}-\Psi_{1,0},\Omega_{R}) = 0.
\end{equation}
Let us consider problem \eqref{eq-hom} with $\vartheta=1$, where $\alpha\geq 0$ is thought as a homotopic parameter. Since solutions of \eqref{eq-hom} are non-negative and thus solve \eqref{eq-lem-deg0}, hypothesis $(H_{2})$ ensures that the degree $\mathrm{deg}_{\mathrm{LS}}(\mathrm{Id}_{X}-\Psi_{1,\alpha},\Omega_{R})$ is well-defined for every $\alpha\geq0$. The homotopy invariance property of the degree and hypothesis $(H_{3})$ finally ensure that
\begin{equation*}
\mathrm{deg}_{\mathrm{LS}}(\mathrm{Id}_{X}-\Psi_{1,0},\Omega_{R}) 
= \mathrm{deg}_{\mathrm{LS}}(\mathrm{Id}_{X}-\Psi_{1,\alpha_{0}},\Omega_{R}) = 0.
\end{equation*}
The claim follows.

As a final step, from \eqref{eq-deg1} and \eqref{eq-deg0}, by the additivity property of the degree, we infer that
\begin{equation*} 
\mathrm{deg}_{\mathrm{LS}}(\mathrm{Id}_{X}-\Psi_{1,0},\Omega_{R}\setminus\overline{\Omega_{r}}) \neq 0.
\end{equation*}
Finally, the existence property of the degree ensures the existence of a solution $w\in \Omega_{R}\setminus\overline{\Omega_{r}}$ of \eqref{eq-fix} with $\vartheta=1$ and $\alpha=0$. Therefore, as explained in the second step, we end up with a solution $u(t)$ of the boundary value problem \eqref{eq-hom} with $\vartheta=1$ and $\alpha=0$, satisfying $r < \| u \|_{\infty} < R$; by a maximum principle argument, $u(t)$ is non-negative and hence solves \eqref{pb-main}.
\end{proof}

\subsection{Proof of Theorem~\ref{th-main-ex}}\label{section-3.2}

We are going to apply Lemma~\ref{lem-deg} with $\varphi$ as in \eqref{def-phi}.

\medskip
\noindent
\textbf{Verification of $(H_{1})$.}
We prove that there exists $r>0$ such that, for every $\vartheta\in \mathopen{]}0,1\mathclose{]}$, problem \eqref{eq-lem-deg1} does not have solutions with $0<\|u\|_{\infty} \leq r$.

By contradiction, let $(u_{n}(t))_{n}$ be a sequence of solutions of \eqref{eq-lem-deg1} for $\vartheta=\vartheta_{n}$ satisfying $0<\|u_{n}\|_{\infty}=: r_{n}\to 0$. The functions
\begin{equation*}
v_{n}(t): = \dfrac{u_{n}(t)}{r_{n}}
\end{equation*}
solve
\begin{equation*}
\Biggl{(} \dfrac{v_{n}'}{\sqrt{1-(u_{n}')^{2}}}\Biggr{)}' + \vartheta_{n} a(t) q(u_{n}(t)) v_{n} = 0,  \qquad \mathfrak{B}(v_{n})=0,
\end{equation*}
where $q(u) = g(u)/u$ for $u > 0$ and $q(0) = 0$. Multiplying the above equation by $v_{n}$ and integrating by parts on $\mathopen{[}0,T\mathclose{]}$, we obtain
\begin{equation*}
\int_{0}^{T} (v_{n}'(t))^{2} \,\mathrm{d}t \leq \int_{0}^{T} \dfrac{(v_{n}'(t))^{2}}{\sqrt{1-(u_{n}'(t))^{2}}} \, \mathrm{d}t = \vartheta_{n} \int_{0}^{T} a(t)q(u_{n}(t))(v_{n}(t))^{2} \,\mathrm{d}t.
\end{equation*}
Since $\| v_{n} \|_{\infty} \leq 1$, by using the first condition in $(g_{0})$, we deduce that 
$\| v_{n} \|_{L^2(0,T)} \to 0$. As a consequence, $v_{n}(t) \to 1$ uniformly in $t \in \mathopen{[}0,T\mathclose{]}$. 

On the other hand, an integration of the equation for $u_{n}$ on $\mathopen{[}0,T\mathclose{]}$ yields
\begin{equation*}
0 = \int_{0}^{T} a(t) g(u_{n}(t))\,\mathrm{d}t = \int_{0}^{T} a(t) g(r_{n})\,\mathrm{d}t + \int_{0}^{T} a(t)\bigl{(}g(r_{n} v_{n}(t)) - g(r_{n})\bigr{)}\,\mathrm{d}t,
\end{equation*}
so that, dividing by $g(r_{n}) > 0$,
\begin{equation*}
0 < - \int_{0}^{T} a(t)\,\mathrm{d}t \leq \|a\|_{L^{1}(0,T)} \sup_{t\in \mathopen{[}0,T\mathclose{]}}\biggl{|}\dfrac{g(r_{n} v_{n}(t))}{g(r_{n})} - 1\biggr{|}.
\end{equation*}
Since $v_{n}(t) \to 1$ uniformly, by exploiting the second condition in $(g_{0})$, we conclude that the right-hand side of the above inequality tends to zero, a contradiction. Condition $(H_{1})$ is thus verified and we can fix the constant $r$.

\medskip

In order to verify condition $(H_{2})$, some preliminary arguments are needed.
We first define, for $i=1,\ldots,m$, the function 
\begin{equation*}
\gamma_{i}(\delta) = \min_{t \in \mathopen{[}\sigma_{i}+\delta,\tau_{i}\mathclose{]}}\| a\|_{L^1(t-\delta,t)}, \qquad \delta \in \biggl{[}0,\frac{\tau_{i}-\sigma_{i}}{4}\biggr{]},
\end{equation*}
in such a way that $\gamma_{i}((\tau_{i}-\sigma_{i})/4) = A_{i}$ (cf.~\eqref{def-Ai}). 
An easy argument shows that $\gamma_{i}$ is increasing and continuous. Therefore, from condition $(g_{\mathrm{SE}})$
we infer the existence of $\delta_{i} \in \mathopen{]}0,(\tau_{i}-\sigma_{i})/4 \mathclose{[}$ such that
\begin{equation*}
\liminf_{u \to +\infty}\frac{g(u)}{G(u)} > \dfrac{\|a^{-}\|_{L^{\infty}(0,T)}}{\gamma_{i}(\delta_{i})},
\end{equation*}
for every $i =1,\ldots,m$.
Let us now fix 
\begin{equation}\label{delta-eps}
\varepsilon \in \biggl{]}0,\dfrac{\tau_{i}-\sigma_{i} - 4\delta_{i}}{\tau_{i}-\sigma_{i} - 2\delta_{i}}\biggr{[}
\end{equation}
and define
\begin{equation*}
\beta_{i} = \varepsilon \biggl{(} \delta_{i} - \frac{\tau_{i}-\sigma_{i}}{2}\biggr{)} + \frac{\tau_{i}-\sigma_{i}}{2} - 2\delta_{i}.
\end{equation*}
Notice that, in view of \eqref{delta-eps}, it holds that $\beta_{i} > 0$.
Then, we use once more condition $(g_{\mathrm{SE}})$ to chose $R^{*} > 0$ such that
the estimates
\begin{equation}\label{eq-R1}
- \gamma_{i}(\delta_{i}) \min_{s\in \mathopen{[}\rho-\delta_{i},\rho\mathclose{]}} g(s) \leq \varphi(-1+\varepsilon)
\end{equation}
and
\begin{equation}\label{eq-R2}
\dfrac{g(\rho-\delta_{i})}{G((\rho-\delta_{i}) - \beta_{i})} > \dfrac{\|a^{-}\|_{L^{\infty}(0,T)}}{\gamma_{i}(\delta_{i})} - \dfrac{\varphi(-1+\varepsilon)}{\gamma_{i}(\delta_{i}) G((\rho-\delta_{i}) -\beta_{i})}
\end{equation}
hold for all $\rho\geq R^{*}$ and $i=1,\ldots,m$. 
Notice that the above choice is possible since $g(u) \to +\infty$ for $u \to +\infty$ (by \eqref{g-infty}),
$G(u)$ is monotone increasing, and $G(u) \to +\infty$ for $u \to +\infty$ (by \eqref{G-infty}).

Finally, let $v(t)$ be the indicator function of the set $\bigcup_{i=1}^{m} I^{+}_{i}$. 
We are going to verify conditions $(H_{2})$ and $(H_{3})$ for
\begin{equation*}
R := \max \bigl{\{} R^{*}, \hat{R}\bigr{\}} + 2 \max_{i=1,\ldots,m}\delta_{i} + T,
\end{equation*}
where $\hat{R}$ is introduced in hypothesis $(g_{\nearrow})$.

\medskip
\noindent
\textbf{Verification of $(H_{2})$.}
For this verification we recall the convention on the nodal behavior of $a(t)$ made in Remark~\ref{rem-nodal} for the periodic boundary conditions.

By contradiction, we assume that $u(t)$ is a solution of \eqref{eq-lem-deg0} such that $\|u\|_{\infty} = R$.
Since $u(t)$ is convex in the negativity intervals $I_{i}^{-}$, we can assume that the maximum is reached in a point $\hat{t}_i\in I^{+}_{i}=\mathopen{[}\sigma_{i},\tau_{i}\mathclose{]}$, for some $i=1,\ldots,m$, and that $u'(\hat{t}_i)=0$. Notice that, in case of Neumann boundary conditions, the possibility that the maximum is reached for $t = 0$ (respectively, $t = T$) is excluded if $a(t) \leq 0$ in a right neighborhood of $t = 0$ (respectively, left neighborhood of $t = T$).

Clearly at least one of the following situations occurs
\begin{equation}\label{2cases}
\hat{t}_i\in\biggl{[}\sigma_{i},\frac{\sigma_{i}+\tau_{i}}{2}\biggr{]}
\qquad \text{ or } \qquad
\hat{t}_i\in\biggl{[}\frac{\sigma_{i}+\tau_{i}}{2},\tau_{i}\biggr{]}.
\end{equation}
Our goal is to show that in the first case we have
\begin{equation}\label{eq-case1}
u'(t)\leq -1+\varepsilon, \quad \text{for every $t \in J^{+}$,}
\end{equation}
while in the second one we have
\begin{equation}\label{eq-case2}
u'(t)\geq 1-\varepsilon, \quad \text{for every $t \in J^{-}$},
\end{equation}
where the intervals $J^{\pm}$ are defined by 
\begin{equation*}
J^{+} = \mathopen{[}\hat{t}_i+\delta_{i},\hat{t}_i+\delta_{i}+T\mathclose{]},
\qquad
J^{-}= \mathopen{[}\hat{t}_i-\delta_{i}-T,\hat{t}_i-\delta_{i}\mathclose{]},
\end{equation*}
in the case of periodic boundary conditions (we agree that the function $u(t)$ is extended to the whole real line by $T$-periodicity)
and by
\begin{equation*}
J^{+} = \mathopen{[}\hat{t}_i+\delta_{i},T\mathclose{]},
\qquad
J^{-} = \mathopen{[}0,\hat{t}_i-\delta_{i}\mathclose{]},
\end{equation*}
in the case of Neumann boundary conditions. Hence, both for periodic and Neumann boundary conditions a contradiction is reached.
For further convenience, let us observe that, due to $\| u' \|_{\infty} < 1$ and the fact that the distances of 
$\hat{t}_i$ from $ \sup J^+$ or $\inf J^-$ are less than or equal to $\max_i\delta_{i} + T$ , it holds that
\begin{equation}\label{eq-fondamentale}
u(t) \geq \max \bigl{\{} R^{*}, \hat{R}\bigr{\}} + \max_{i=1,\ldots,m}\delta_{i}, \quad \text{for every $t \in J^{-} \cup J^{+}$.}
\end{equation}

We provide the detailed argument in the first case of \eqref{2cases}. The other case can be treated in a symmetric way.

Firstly, we prove that
\begin{equation}\label{eq-hat-t}
u'(\hat{t}_i+\delta_{i})\leq -1+\varepsilon.
\end{equation}
Indeed, notice that, since $|u'(t)| < 1$ for every $t \in \mathopen{[}0,T\mathclose{]}$,
then $u(t) \geq R - \delta_{i}$ for $t \in \mathopen{[}\hat{t}_i,\hat{t}_i+\delta_{i}\mathclose{]}$.
An integration of the equation in such an interval then leads to
\begin{align}
\varphi(u'(\hat{t}_i+\delta_{i})) &= \varphi(u'(\hat{t}_i)) + \int_{\hat{t}_i}^{\hat{t}_i+\delta_{i}} \bigl{[}\varphi(u'(t))\bigr{]}' \,\mathrm{d}t 
\\ &= -  \int_{\hat{t}_i}^{\hat{t}_i+\delta_{i}} \bigl{(}a(t)g(u(t))+\alpha v(t)\bigr{)} \,\mathrm{d}t 
\label{eq-for-iteration}
\\ &\leq -\|a\|_{L^{1}(\hat{t}_i,\hat{t}_i+\delta_{i})} \min_{t\in \mathopen{[}\hat{t}_i,\hat{t}_i+\delta_{i}\mathclose{]}} g(u(t))
\\ &\leq - \gamma_{i}(\delta_{i}) \min_{s\in \mathopen{[}R-\delta_{i},R\mathclose{]}} g(s)
\\ &\leq \varphi(-1+\varepsilon),
\end{align}
where the last inequality follows from \eqref{eq-R1}. Therefore, \eqref{eq-hat-t} follows.

From this estimate, together with the fact that $t\mapsto\varphi(u'(t))$ is monotone decreasing in $I^{+}_{i}$, we have
\begin{equation}\label{eq-hat-t-2}
u'(t)\leq-1+\varepsilon, \quad \text{for all $t\in\mathopen{[}\hat{t}_i+\delta_{i},\tau_{i}\mathclose{]}$.}
\end{equation}

The second step consists in finding upper bounds for both $u(\tau_{i})$ and $\varphi(u'(\tau_{i}))$.
The one for $u(\tau_{i})$ is easily achieved: indeed, from \eqref{eq-hat-t-2} and recalling that $u'(t)\leq0$ for every $t\in\mathopen{[}\hat{t}_i,\hat{t}_i+\delta_{i}\mathclose{]}$, we immediately find that
\begin{align}
u(\tau_{i}) &= u(\hat{t}_i) + \int_{\hat{t}_i}^{\tau_{i}} u'(t) \,\mathrm{d}t  = R + \int_{\hat{t}_i}^{\hat{t}_i+\delta_{i}} u'(t) \,\mathrm{d}t + \int_{\hat{t}_i+\delta_{i}}^{\tau_{i}} u'(t) \,\mathrm{d}t
\\ &\leq R + (-1+\varepsilon) (\tau_{i}-\hat{t}_i-\delta_{i}).
\label{estimatetau1}
\end{align}
As for the estimate on $\varphi(u'(\tau_{i}))$, we proceed as in the proof of \eqref{eq-hat-t}
to obtain
\begin{equation}\label{estimatetau2}
\varphi(u'(\tau_{i})) \leq  \varphi(u'(\hat{t}_i+\delta_{i})) \leq - \gamma_{i}(\delta_{i}) \min_{s\in \mathopen{[}R-\delta_{i},R\mathclose{]}} g(s) \leq - \gamma_{i}(\delta_{i}) \, g(R-\delta_{i}),
\end{equation}
where the last inequality follows from $(g_{\nearrow})$, taking into account that $R-\delta_i \geq \hat{R}$.

Third, let us consider the subsequent interval $I_{i}^{-}=\mathopen{[}\tau_{i},\sigma_{i+1}\mathclose{]}$
(notice that, in the case of Neumann boundary conditions, if there is no such an interval the proof is already completed; 
in the periodic case, instead, this possibility is excluded in view of Remark~\ref{rem-nodal}). We claim that
\begin{equation}\label{eq-hat-t-3}
u'(t)\leq-1+\varepsilon, \quad \text{for all $t\in\mathopen{[}\tau_{i},\sigma_{i+1}\mathclose{]}$.}
\end{equation}
Let $\mathopen{[}\tau_{i},t^{*}\mathclose{]}$ be the maximal interval in $\mathopen{[}\tau_{i},\sigma_{i+1}\mathclose{]}$ such that $u'(t)\leq-1+\varepsilon$ for all $t\in\mathopen{[}\tau_{i},t^{*}\mathclose{]}$.
Then, using \eqref{estimatetau1} we find
\begin{align*}
u(t) &= u(\tau_{i}) + \int_{\tau_{i}}^{t} u'(\xi) \,\mathrm{d}\xi 
\\ &\leq R + (-1+\varepsilon) (\tau_{i}-\hat{t}_i-\delta_{i}) + (-1+\varepsilon) (t-\tau_{i})
\\ &= R + (-1+\varepsilon) (t-\hat{t}_i-\delta_{i}),
\end{align*}
for all $t\in\mathopen{[}\tau_{i},t^{*}\mathclose{]}$. 
By contradiction, suppose that $t^{*}<\sigma_{i+1}$.

Then, integrating the equation (recall that $v \equiv 0$ on $I^{-}_{i}$) and using \eqref{estimatetau2} we obtain
\begin{align*}
\varphi(-1+\varepsilon)
&=\varphi(u'(t^{*}))
= \varphi(u'(\tau_{i})) + \int_{\tau_{i}}^{t^{*}} a(t)g(u(t)) \,\mathrm{d}t 
\\ &\leq \varphi(u'(\tau_{i})) + \|a^{-}\|_{L^{\infty}(0,T)}\int_{\tau_{i}}^{t^{*}} g(R +(-1+\varepsilon)(t-\hat{t}_i-\delta_{i})) \,\mathrm{d}t 
\\ &\leq -\gamma_{i}(\delta_{i}) \, g(R -\delta_{i}) + \|a^{-}\|_{L^{\infty}(0,T)} \bigl{[} G(R + (-1+\varepsilon) (\tau_{i}-\hat{t}_i-\delta_{i}))
\\ &\hspace{120pt} - G(R + (-1+\varepsilon) (t^{*}-\hat{t}_i-\delta_{i}) ) \bigr{]}
\\ &\leq -\gamma_{i}(\delta_{i}) \, g(R -\delta_{i}) + \|a^{-}\|_{L^{\infty}(0,T)} G(R + (-1+\varepsilon) (\tau_{i}-\hat{t}_i-\delta_{i}))
\\ &\leq -\gamma_{i}(\delta_{i}) \, g(R -\delta_{i})  + \|a^{-}\|_{L^{\infty}(0,T)} G((R -\delta_{i}) - \beta_{i}),
\end{align*}
where the last inequality comes from the facts that $G(u)$ is increasing and $\tau_{i}-\hat{t}_i\geq (\tau_{i}-\sigma_{i})/2$.
Then, 
\begin{equation}\label{eq-3.10}
\dfrac{g(R-\delta_{i})}{G((R-\delta_{i}) - \beta_{i})} \leq \dfrac{\|a^{-}\|_{L^{\infty}(0,T)}}{\gamma_{i}(\delta_{i})} - 
\dfrac{\varphi(-1+\varepsilon)}{\gamma_{i}(\delta_{i}) G((R-\delta_{i}) - \beta_{i})},
\end{equation}
a contradiction with respect to \eqref{eq-R2}. Then, $t^* = \sigma_{i+1}$
and \eqref{eq-hat-t-3} is proved.

Summing up, by \eqref{eq-hat-t-2} and \eqref{eq-hat-t-3}, we have that
\begin{equation*}
u'(t)\leq-1+\varepsilon, \quad \text{for all $t\in\mathopen{[}\hat{t}_i,\sigma_{i+1}\mathclose{]}$.}
\end{equation*}
If $\sigma_{i+1}=T$ and Neumann boundary conditions are taken into account, the proof is concluded.
Otherwise, we first observe that, by convexity,
\begin{equation*}
u'(t)\leq-1+\varepsilon, \quad \text{for all $t\in\mathopen{[}\sigma_{i+1},\tau_{i+1}\mathclose{]}$},
\end{equation*}
(if $\sigma_{i+1}=T$ and periodic boundary conditions are considered, both the function $u(t)$ and the 
weight $a(t)$ are extended by $T$-periodicity).
If $\tau_{i+1} = T$ and Neumann boundary conditions are taken into account, the proof is concluded.
Otherwise, we claim that 
\begin{equation}\label{estimatei+1}
u'(t)\leq-1+\varepsilon, \quad \text{for all $t\in\mathopen{[}\tau_{i+1},\sigma_{i+2}\mathclose{]}$}.
\end{equation}
This can be proved in a similar manner as before. More precisely, we first prove that
\begin{equation}\label{estimatetau1bis}
u(\tau_{i+1}) \leq R_{i+1} + (-1+\varepsilon) (\tau_{i+1}-\hat{t}_{i+1}-\delta_{i+1})
\end{equation}
and
\begin{equation}\label{estimatetau2bis}
\varphi(u'(\tau_{i+1}))  \leq -\gamma_{i+1}(\delta_{i+1}) \, g(R_{i+1}-\delta_{i+1}),
\end{equation}
where 
\begin{equation*}
\hat{t}_{i+1} = \sigma_{i+1} \quad \text{ and } \quad R_{i+1} = u(\hat{t}_{i+1}).
\end{equation*}
The above estimates are analogous to \eqref{estimatetau1} and \eqref{estimatetau2}, respectively, and can be proved with the very same arguments: indeed, $\hat{t}_{i+1}$ is the maximum point for $u(t)$ in the interval $I^+_{i+1}$ (in formula \eqref{eq-for-iteration}, the equality is replaced by an inequality, since $u'(\hat{t}_{i+1}) \leq 0$) and, as a consequence of \eqref{eq-fondamentale},
$R_{i+1} \geq \max\{R^*,\hat{R}\} + \delta_{i+1}$.
Using \eqref{estimatetau1bis} and \eqref{estimatetau2bis},
the proof of \eqref{estimatei+1} can be done with the same contradiction argument as before: more precisely, we achieve estimate \eqref{eq-3.10} with $R_{i+1}$ in place of $R$ and $i+1$ in place of $i$, thus contradicting \eqref{eq-R2} since $R_{i+1} \geq R^*$.

Arguing in an iterative way \eqref{eq-case1} can be established and the proof of $(H_{2})$ is thus completed.

\medskip
\noindent
\textbf{Verification of $(H_{3})$.}
Let
\begin{equation*}
\alpha_{0} > \dfrac{\|a\|_{L^{1}(0,T)} \max_{s\in\mathopen{[}0,R\mathclose{]}} g(s)}{\|v\|_{L^{1}(0,T)}}.
\end{equation*}
Looking for solution of \eqref{eq-lem-deg0} with $\|u\|_{\infty}\leq R$, we integrate equation \eqref{eq-lem-deg0} on $\mathopen{[}0,T\mathclose{]}$ and pass to the absolute value, thus obtaining
\begin{equation*}
\alpha \| v \|_{L^{1}(0,T)} \leq \| a \|_{L^{1}(0,T)} \max_{s \in \mathopen{[}0,R\mathclose{]}} g(s).
\end{equation*}
It is clear that for $\alpha\geq\alpha_{0}$ such solutions do not exist.

\subsubsection*{Conclusion of the proof}
By applying Lemma~\ref{lem-deg} we obtain a solution of problem \eqref{eq-main-phi}. A direct application of a strong maximum principle (see, for instance, \cite[Theorem~A.2]{BoFe-PP}) ensures that the solution is positive. The proof of Theorem~\ref{th-main-ex} is thus completed.
\hfill\qed

\begin{remark}\label{rem-3.1}
The arguments used in the verification of hypothesis $(H_{2})$ show, in particular, that a quite unexpected property of the differential equation \eqref{eq-main-phi} holds true: roughly speaking, its ``large solutions'' 
have a $\wedge$-shaped graph, with slope approaching the values $\pm 1$ (indeed, \eqref{eq-case1} and \eqref{eq-case2} can be obtained for arbitrarily small values of $\varepsilon > 0$, up to choosing the maximum value sufficiently large).
This is the key point of the proof, and it is of course a consequence of assumption $(g_{\mathrm{SE}})$: we refer to Figure~\ref{fig-03} for a graphical explanation and more comments about this. It is interesting to point out that, for the semilinear equation \eqref{eq-semi} with superlinear growth of $g(u)$ at infinity, condition $(H_{2})$ is also true (compare with condition 
$(H_{R})$ in \cite[Theorem~2.1]{FeZa-15ade}). However, the proof of its validity relies on very different arguments, namely, a Sturm comparison technique involving only the equation on the positivity intervals (cf.~\cite[Lemma~6.2]{FeZa-15ade}). On the contrary, here the
crucial estimate has to be achieved in the negativity intervals (compare again with Figure~\ref{fig-03}). Hence, 
even if Theorem~\ref{th-main-ex} can be interpreted as a Minkowski-curvature analogue of more classical results for semilinear indefinite equations with superlinear nonlinearities, the dynamics of the corresponding differential equations is drastically different. 
\hfill$\lhd$
\end{remark}

\begin{figure}[htb]
\centering
\begin{tikzpicture}[scale=1]
\begin{axis}[
  tick label style={font=\scriptsize},
  axis y line=middle, 
  axis x line=middle,
  xtick={-2,-1,0,1,2},
  ytick={0,1},
  xticklabels={},
  yticklabels={$0$},
  xlabel={\small $t$},
  ylabel={\small $u(t)$},
every axis x label/.style={
    at={(ticklabel* cs:1.0)},
    anchor=west,
},
every axis y label/.style={
    at={(ticklabel* cs:1.0)},
    anchor=south,
},
  width=6.7cm,
  height=6cm,
  xmin=-2.6,
  xmax=2.7,
  ymin=-0.7,
  ymax=2.7]
\addplot [color=BF-blue!55!BF-green,line width=0.7pt,smooth] coordinates {(-2.,0.382545) (-1.9,0.360217) (-1.8,0.343123) (-1.7,0.330876) (-1.6,0.323198) (-1.5,0.319912) (-1.4,0.320945) (-1.3,0.326319) (-1.2,0.336157) (-1.1,0.350684) (-1.,0.37023) (-0.9,0.39171) (-0.8,0.411637) (-0.7,0.429818) (-0.6,0.44606) (-0.5,0.460178) (-0.4,0.471998) (-0.3,0.481366) (-0.2,0.488156) (-0.1,0.49227) (0.,0.493648) (0.1,0.49227) (0.2,0.488156) (0.3,0.481366) (0.4,0.471998) (0.5,0.460178) (0.6,0.44606) (0.7,0.429818) (0.8,0.411637) (0.9,0.39171) (1.,0.37023) (1.1,0.350684) (1.2,0.336157) (1.3,0.326319) (1.4,0.320945) (1.5,0.319912) (1.6,0.323198) (1.7,0.330876) (1.8,0.343123) (1.9,0.360217) (2.,0.382545)};
\addplot [color=BF-red,line width=0.7pt,smooth] coordinates {(-2.,0.267815) (-1.9,0.269305) (-1.8,0.273804) (-1.7,0.281406) (-1.6,0.292269) (-1.5,0.306621) (-1.4,0.324769) (-1.3,0.347106) (-1.2,0.374115) (-1.1,0.406371) (-1.,0.444529) (-0.9,0.484977) (-0.8,0.523365) (-0.7,0.559229) (-0.6,0.59205) (-0.5,0.621265) (-0.4,0.646283) (-0.3,0.666517) (-0.2,0.681426) (-0.1,0.690567) (0.,0.693648) (0.1,0.690567) (0.2,0.681426) (0.3,0.666517) (0.4,0.646283) (0.5,0.621265) (0.6,0.59205) (0.7,0.559229) (0.8,0.523365) (0.9,0.484977) (1.,0.444529) (1.1,0.406371) (1.2,0.374115) (1.3,0.347106) (1.4,0.324769) (1.5,0.306621) (1.6,0.292269) (1.7,0.281406) (1.8,0.273804) (1.9,0.269305) (2.,0.267815)};
\addplot [color=BF-blue!60!BF-green,line width=0.7pt,smooth] coordinates {(-2.,0.00570904) (-1.9,0.0485909) (-1.8,0.0915515) (-1.7,0.134769) (-1.6,0.178531) (-1.5,0.223239) (-1.4,0.269403) (-1.3,0.317642) (-1.2,0.368673) (-1.1,0.423281) (-1.,0.482275) (-0.9,0.543004) (-0.8,0.601961) (-0.7,0.658537) (-0.6,0.711941) (-0.5,0.761167) (-0.4,0.804946) (-0.3,0.841747) (-0.2,0.869852) (-0.1,0.887578) (0.,0.893648) (0.1,0.887578) (0.2,0.869852) (0.3,0.841747) (0.4,0.804946) (0.5,0.761167) (0.6,0.711941) (0.7,0.658537) (0.8,0.601961) (0.9,0.543004) (1.,0.482275) (1.1,0.423281) (1.2,0.368673) (1.3,0.317642) (1.4,0.269403) (1.5,0.223239) (1.6,0.178531) (1.7,0.134769) (1.8,0.0915515) (1.9,0.0485909) (2.,0.00570904)};
\addplot [color=BF-blue!65!BF-green,line width=0.7pt,smooth] coordinates {(-2.,-0.211618) (-1.9,-0.140983) (-1.8,-0.0703486) (-1.7,0.000286047) (-1.6,0.0709266) (-1.5,0.141651) (-1.4,0.212671) (-1.3,0.284342) (-1.2,0.357138) (-1.1,0.431631) (-1.,0.508424) (-0.9,0.586152) (-0.8,0.66282) (-0.7,0.737956) (-0.6,0.810881) (-0.5,0.880588) (-0.4,0.945548) (-0.3,1.0034) (-0.2,1.0506) (-0.1,1.08233) (0.,1.09365) (0.1,1.08233) (0.2,1.0506) (0.3,1.0034) (0.4,0.945548) (0.5,0.880588) (0.6,0.810881) (0.7,0.737956) (0.8,0.66282) (0.9,0.586152) (1.,0.508424) (1.1,0.431631) (1.2,0.357138) (1.3,0.284342) (1.4,0.212671) (1.5,0.141651) (1.6,0.0709266) (1.7,0.000286047) (1.8,-0.0703486) (1.9,-0.140983) (2.,-0.211618)};
\addplot [color=BF-blue!70!BF-green,line width=0.7pt,smooth] coordinates {(-2.,-0.314131) (-1.9,-0.227818) (-1.8,-0.141504) (-1.7,-0.0551898) (-1.6,0.0311241) (-1.5,0.117449) (-1.4,0.203851) (-1.3,0.290475) (-1.2,0.377534) (-1.1,0.465302) (-1.,0.554078) (-0.9,0.643269) (-0.8,0.731967) (-0.7,0.819929) (-0.6,0.906783) (-0.5,0.991929) (-0.4,1.07433) (-0.3,1.15204) (-0.2,1.2211) (-0.1,1.27313) (0.,1.29365) (0.1,1.27313) (0.2,1.2211) (0.3,1.15204) (0.4,1.07433) (0.5,0.991929) (0.6,0.906783) (0.7,0.819929) (0.8,0.731967) (0.9,0.643269) (1.,0.554078) (1.1,0.465302) (1.2,0.377534) (1.3,0.290475) (1.4,0.203851) (1.5,0.117449) (1.6,0.0311241) (1.7,-0.0551898) (1.8,-0.141504) (1.9,-0.227818) (2.,-0.314131)};
\addplot [color=BF-blue!75!BF-green,line width=0.7pt,smooth] coordinates {(-2.,-0.306183) (-1.9,-0.211843) (-1.8,-0.117504) (-1.7,-0.023165) (-1.6,0.0711745) (-1.5,0.165524) (-1.4,0.259915) (-1.3,0.354409) (-1.2,0.449087) (-1.1,0.544055) (-1.,0.639423) (-0.9,0.734956) (-0.8,0.830295) (-0.7,0.925348) (-0.6,1.01996) (-0.5,1.11389) (-0.4,1.20668) (-0.3,1.29738) (-0.2,1.38366) (-0.1,1.45804) (0.,1.49365) (0.1,1.45804) (0.2,1.38366) (0.3,1.29738) (0.4,1.20668) (0.5,1.11389) (0.6,1.01996) (0.7,0.925348) (0.8,0.830295) (0.9,0.734956) (1.,0.639423) (1.1,0.544055) (1.2,0.449087) (1.3,0.354409) (1.4,0.259915) (1.5,0.165524) (1.6,0.0711745) (1.7,-0.023165) (1.8,-0.117504) (1.9,-0.211843) (2.,-0.306183)};
\addplot [color=BF-blue!80!BF-green,line width=0.7pt,smooth] coordinates {(-2.,-0.213756) (-1.9,-0.115784) (-1.8,-0.0178114) (-1.7,0.0801613) (-1.6,0.178137) (-1.5,0.276123) (-1.4,0.374134) (-1.3,0.472193) (-1.2,0.570325) (-1.1,0.668563) (-1.,0.76694) (-0.9,0.865373) (-0.8,0.963743) (-0.7,1.06202) (-0.6,1.16015) (-0.5,1.25807) (-0.4,1.3556) (-0.3,1.45245) (-0.2,1.54773) (-0.1,1.63796) (0.,1.69365) (0.1,1.63796) (0.2,1.54773) (0.3,1.45245) (0.4,1.3556) (0.5,1.25807) (0.6,1.16015) (0.7,1.06202) (0.8,0.963743) (0.9,0.865373) (1.,0.76694) (1.1,0.668563) (1.2,0.570325) (1.3,0.472193) (1.4,0.374134) (1.5,0.276123) (1.6,0.178137) (1.7,0.0801613) (1.8,-0.0178114) (1.9,-0.115784) (2.,-0.213756)};
\addplot [color=BF-blue!85!BF-green,line width=0.7pt,smooth] coordinates {(-2.,-0.0670807) (-1.9,0.0323081) (-1.8,0.131697) (-1.7,0.231087) (-1.6,0.33048) (-1.5,0.429879) (-1.4,0.529289) (-1.3,0.628715) (-1.2,0.728165) (-1.1,0.827645) (-1.,0.927165) (-0.9,1.0267) (-0.8,1.12622) (-0.7,1.22571) (-0.6,1.32517) (-0.5,1.42456) (-0.4,1.52386) (-0.3,1.62297) (-0.2,1.72165) (-0.1,1.81888) (0.,1.89365) (0.1,1.81888) (0.2,1.72165) (0.3,1.62297) (0.4,1.52386) (0.5,1.42456) (0.6,1.32517) (0.7,1.22571) (0.8,1.12622) (0.9,1.0267) (1.,0.927165) (1.1,0.827645) (1.2,0.728165) (1.3,0.628715) (1.4,0.529289) (1.5,0.429879) (1.6,0.33048) (1.7,0.231087) (1.8,0.131697) (1.9,0.0323081) (2.,-0.0670807)};
\addplot [color=BF-blue!90!BF-green,line width=0.7pt,smooth] coordinates {(-2.,0.109217) (-1.9,0.209065) (-1.8,0.308914) (-1.7,0.408763) (-1.6,0.508614) (-1.5,0.608467) (-1.4,0.708322) (-1.3,0.808181) (-1.2,0.908046) (-1.1,1.00792) (-1.,1.1078) (-0.9,1.20768) (-0.8,1.30756) (-0.7,1.40744) (-0.6,1.5073) (-0.5,1.60716) (-0.4,1.70699) (-0.3,1.80678) (-0.2,1.90647) (-0.1,2.00584) (0.,2.09365) (0.1,2.00584) (0.2,1.90647) (0.3,1.80678) (0.4,1.70699) (0.5,1.60716) (0.6,1.5073) (0.7,1.40744) (0.8,1.30756) (0.9,1.20768) (1.,1.1078) (1.1,1.00792) (1.2,0.908046) (1.3,0.808181) (1.4,0.708322) (1.5,0.608467) (1.6,0.508614) (1.7,0.408763) (1.8,0.308914) (1.9,0.209065) (2.,0.109217)};
\addplot [color=BF-blue!95!BF-green,line width=0.7pt,smooth] coordinates {(-2.,0.299519) (-1.9,0.399489) (-1.8,0.499458) (-1.7,0.599428) (-1.6,0.699397) (-1.5,0.799368) (-1.4,0.899338) (-1.3,0.99931) (-1.2,1.09928) (-1.1,1.19926) (-1.,1.29923) (-0.9,1.39921) (-0.8,1.49918) (-0.7,1.59916) (-0.6,1.69913) (-0.5,1.7991) (-0.4,1.89906) (-0.3,1.99902) (-0.2,2.09896) (-0.1,2.19885) (0.,2.29365) (0.1,2.19885) (0.2,2.09896) (0.3,1.99902) (0.4,1.89906) (0.5,1.7991) (0.6,1.69913) (0.7,1.59916) (0.8,1.49918) (0.9,1.39921) (1.,1.29923) (1.1,1.19926) (1.2,1.09928) (1.3,0.99931) (1.4,0.899338) (1.5,0.799368) (1.6,0.699397) (1.7,0.599428) (1.8,0.499458) (1.9,0.399489) (2.,0.299519)};
\addplot [color=BF-blue!100!BF-green,line width=0.7pt,smooth] coordinates {(-2.,0.495765) (-1.9,0.59576) (-1.8,0.695755) (-1.7,0.79575) (-1.6,0.895745) (-1.5,0.99574) (-1.4,1.09573) (-1.3,1.19573) (-1.2,1.29573) (-1.1,1.39572) (-1.,1.49572) (-0.9,1.59571) (-0.8,1.69571) (-0.7,1.7957) (-0.6,1.8957) (-0.5,1.99569) (-0.4,2.09569) (-0.3,2.19568) (-0.2,2.29567) (-0.1,2.39565) (0.,2.49365) (0.1,2.39565) (0.2,2.29567) (0.3,2.19568) (0.4,2.09569) (0.5,1.99569) (0.6,1.8957) (0.7,1.7957) (0.8,1.69571) (0.9,1.59571) (1.,1.49572) (1.1,1.39572) (1.2,1.29573) (1.3,1.19573) (1.4,1.09573) (1.5,0.99574) (1.6,0.895745) (1.7,0.79575) (1.8,0.695755) (1.9,0.59576) (2.,0.495765)};
\addplot [color=gray,line width=0.4pt,dashed] coordinates {(-2,-0.4) (-2,0.7)};
\addplot [color=gray,line width=0.4pt,dashed] coordinates {(2,-0.4) (2,0.7)};
\node[label={270:{\scriptsize{-2}}}] at (axis cs:-2.2,0.04) {};
\node[label={270:{\scriptsize{0}}}] at (axis cs:-0.15,0.04) {};
\node[label={270:{\scriptsize{2}}}] at (axis cs:2.15,0.04) {};
\end{axis}
\end{tikzpicture}
\quad\quad\quad
\begin{tikzpicture}[scale=1]
\begin{axis}[
  tick label style={font=\scriptsize},
  axis y line=middle, 
  axis x line=middle,
  xtick={-2,-1,0,1,2},
  ytick={0,1},
  xticklabels={},
  yticklabels={$0$},
  xlabel={\small $t$},
  ylabel={\small $u(t)$},
every axis x label/.style={
    at={(ticklabel* cs:1.0)},
    anchor=west,
},
every axis y label/.style={
    at={(ticklabel* cs:1.0)},
    anchor=south,
},
  width=6.7cm,
  height=6cm,
  xmin=-2.6,
  xmax=2.7,
  ymin=-0.5,
  ymax=3.8]
\addplot [color=BF-blue!40!BF-green,line width=0.7pt,smooth] coordinates {(-2.,0.422938) (-1.9,0.407215) (-1.8,0.397981) (-1.7,0.395058) (-1.6,0.398388) (-1.5,0.408036) (-1.4,0.42419) (-1.3,0.447141) (-1.2,0.477267) (-1.1,0.514982) (-1.,0.560656) (-0.9,0.609301) (-0.8,0.655409) (-0.7,0.698435) (-0.6,0.737771) (-0.5,0.772754) (-0.4,0.802691) (-0.3,0.82689) (-0.2,0.844713) (-0.1,0.855638) (0.,0.85932) (0.1,0.855638) (0.2,0.844713) (0.3,0.82689) (0.4,0.802691) (0.5,0.772754) (0.6,0.737771) (0.7,0.698435) (0.8,0.655409) (0.9,0.609301) (1.,0.560656) (1.1,0.514982) (1.2,0.477267) (1.3,0.447141) (1.4,0.42419) (1.5,0.408036) (1.6,0.398388) (1.7,0.395058) (1.8,0.397981) (1.9,0.407215) (2.,0.422938)};
\addplot [color=BF-red,line width=0.7pt,smooth] coordinates {(-2.,0.355736) (-1.9,0.358271) (-1.8,0.36593) (-1.7,0.378869) (-1.6,0.397348) (-1.5,0.421721) (-1.4,0.452424) (-1.3,0.489933) (-1.2,0.534711) (-1.1,0.587112) (-1.,0.647273) (-0.9,0.710244) (-0.8,0.770756) (-0.7,0.828135) (-0.6,0.881557) (-0.5,0.930032) (-0.4,0.972401) (-0.3,1.00737) (-0.2,1.03362) (-0.1,1.04994) (0.,1.05548) (0.1,1.04994) (0.2,1.03362) (0.3,1.00737) (0.4,0.972401) (0.5,0.930032) (0.6,0.881557) (0.7,0.828135) (0.8,0.770756) (0.9,0.710244) (1.,0.647273) (1.1,0.587112) (1.2,0.534711) (1.3,0.489933) (1.4,0.452424) (1.5,0.421721) (1.6,0.397348) (1.7,0.378869) (1.8,0.36593) (1.9,0.358271) (2.,0.355736)};
\addplot [color=BF-blue!45!BF-green,line width=0.7pt,smooth] coordinates {(-2.,0.314499) (-1.9,0.332827) (-1.8,0.355321) (-1.7,0.382415) (-1.6,0.414613) (-1.5,0.452469) (-1.4,0.496546) (-1.3,0.547347) (-1.2,0.605223) (-1.1,0.670259) (-1.,0.742196) (-0.9,0.816537) (-0.8,0.888776) (-0.7,0.958235) (-0.6,1.02402) (-0.5,1.08496) (-0.4,1.1395) (-0.3,1.1857) (-0.2,1.22127) (-0.1,1.24387) (0.,1.25164) (0.1,1.24387) (0.2,1.22127) (0.3,1.1857) (0.4,1.1395) (0.5,1.08496) (0.6,1.02402) (0.7,0.958235) (0.8,0.888776) (0.9,0.816537) (1.,0.742196) (1.1,0.670259) (1.2,0.605223) (1.3,0.547347) (1.4,0.496546) (1.5,0.452469) (1.6,0.414613) (1.7,0.382415) (1.8,0.355321) (1.9,0.332827) (2.,0.314499)};
\addplot [color=BF-blue!50!BF-green,line width=0.7pt,smooth] coordinates {(-2.,0.330008) (-1.9,0.356346) (-1.8,0.387163) (-1.7,0.423013) (-1.6,0.464484) (-1.5,0.512148) (-1.4,0.566482) (-1.3,0.627769) (-1.2,0.695993) (-1.1,0.770781) (-1.,0.851422) (-0.9,0.934009) (-0.8,1.01493) (-0.7,1.09359) (-0.6,1.16916) (-0.5,1.24046) (-0.4,1.30579) (-0.3,1.36271) (-0.2,1.40791) (-0.1,1.43746) (0.,1.4478) (0.1,1.43746) (0.2,1.40791) (0.3,1.36271) (0.4,1.30579) (0.5,1.24046) (0.6,1.16916) (0.7,1.09359) (0.8,1.01493) (0.9,0.934009) (1.,0.851422) (1.1,0.770781) (1.2,0.695993) (1.3,0.627769) (1.4,0.566482) (1.5,0.512148) (1.6,0.464484) (1.7,0.423013) (1.8,0.387163) (1.9,0.356346) (2.,0.330008)};
\addplot [color=BF-blue!55!BF-green,line width=0.7pt,smooth] coordinates {(-2.,0.410455) (-1.9,0.434394) (-1.8,0.465061) (-1.7,0.502906) (-1.6,0.548351) (-1.5,0.601686) (-1.4,0.662959) (-1.3,0.73187) (-1.2,0.807743) (-1.1,0.889601) (-1.,0.976312) (-0.9,1.06455) (-0.8,1.15152) (-0.7,1.23674) (-0.6,1.31952) (-0.5,1.3988) (-0.4,1.47293) (-0.3,1.5393) (-0.2,1.59381) (-0.1,1.63072) (0.,1.64396) (0.1,1.63072) (0.2,1.59381) (0.3,1.5393) (0.4,1.47293) (0.5,1.3988) (0.6,1.31952) (0.7,1.23674) (0.8,1.15152) (0.9,1.06455) (1.,0.976312) (1.1,0.889601) (1.2,0.807743) (1.3,0.73187) (1.4,0.662959) (1.5,0.601686) (1.6,0.548351) (1.7,0.502906) (1.8,0.465061) (1.9,0.434394) (2.,0.410455)};
\addplot [color=BF-red,line width=0.7pt,smooth] coordinates {(-2.,0.561712) (-1.9,0.568021) (-1.8,0.586922) (-1.7,0.618296) (-1.6,0.661819) (-1.5,0.716819) (-1.4,0.782186) (-1.3,0.85642) (-1.2,0.937803) (-1.1,1.02464) (-1.,1.11546) (-0.9,1.20747) (-0.8,1.29853) (-0.7,1.38826) (-0.6,1.47613) (-0.5,1.56124) (-0.4,1.64215) (-0.3,1.71632) (-0.2,1.77931) (-0.1,1.82369) (0.,1.84012) (0.1,1.82369) (0.2,1.77931) (0.3,1.71632) (0.4,1.64215) (0.5,1.56124) (0.6,1.47613) (0.7,1.38826) (0.8,1.29853) (0.9,1.20747) (1.,1.11546) (1.1,1.02464) (1.2,0.937803) (1.3,0.85642) (1.4,0.782186) (1.5,0.716819) (1.6,0.661819) (1.7,0.618296) (1.8,0.586922) (1.9,0.568021) (2.,0.561712)};
\addplot [color=BF-blue!60!BF-green,line width=0.7pt,smooth] coordinates {(-2.,0.795753) (-1.9,0.763956) (-1.8,0.753776) (-1.7,0.766142) (-1.6,0.79993) (-1.5,0.852274) (-1.4,0.919334) (-1.3,0.997199) (-1.2,1.08254) (-1.1,1.17287) (-1.,1.26645) (-0.9,1.36094) (-0.8,1.45473) (-0.7,1.54753) (-0.6,1.63891) (-0.5,1.72817) (-0.4,1.8141) (-0.3,1.89444) (-0.2,1.96476) (-0.1,2.01643) (0.,2.03628) (0.1,2.01643) (0.2,1.96476) (0.3,1.89444) (0.4,1.8141) (0.5,1.72817) (0.6,1.63891) (0.7,1.54753) (0.8,1.45473) (0.9,1.36094) (1.,1.26645) (1.1,1.17287) (1.2,1.08254) (1.3,0.997199) (1.4,0.919334) (1.5,0.852274) (1.6,0.79993) (1.7,0.766142) (1.8,0.753776) (1.9,0.763956) (2.,0.795753)};
\addplot [color=BF-blue!65!BF-green,line width=0.7pt,smooth] coordinates {(-2.,1.09319) (-1.9,1.02176) (-1.8,0.969983) (-1.7,0.947399) (-1.6,0.959687) (-1.5,1.00362) (-1.4,1.07009) (-1.3,1.15036) (-1.2,1.23858) (-1.1,1.33136) (-1.,1.4268) (-0.9,1.52294) (-0.8,1.61855) (-0.7,1.71342) (-0.6,1.80722) (-0.5,1.89941) (-0.4,1.98898) (-0.3,2.07404) (-0.2,2.15045) (-0.1,2.20897) (0.,2.23244) (0.1,2.20897) (0.2,2.15045) (0.3,2.07404) (0.4,1.98898) (0.5,1.89941) (0.6,1.80722) (0.7,1.71342) (0.8,1.61855) (0.9,1.52294) (1.,1.4268) (1.1,1.33136) (1.2,1.23858) (1.3,1.15036) (1.4,1.07009) (1.5,1.00362) (1.6,0.959687) (1.7,0.947399) (1.8,0.969983) (1.9,1.02176) (2.,1.09319)};
\addplot [color=BF-blue!70!BF-green,line width=0.7pt,smooth] coordinates {(-2.,1.39446) (-1.9,1.30468) (-1.8,1.22412) (-1.7,1.16347) (-1.6,1.14105) (-1.5,1.16786) (-1.4,1.23108) (-1.3,1.31283) (-1.2,1.40318) (-1.1,1.4977) (-1.,1.5944) (-0.9,1.69165) (-0.8,1.7885) (-0.7,1.88478) (-0.6,1.98026) (-0.5,2.07449) (-0.4,2.16671) (-0.3,2.2553) (-0.2,2.33664) (-0.1,2.40139) (0.,2.4286) (0.1,2.40139) (0.2,2.33664) (0.3,2.2553) (0.4,2.16671) (0.5,2.07449) (0.6,1.98026) (0.7,1.88478) (0.8,1.7885) (0.9,1.69165) (1.,1.5944) (1.1,1.4977) (1.2,1.40318) (1.3,1.31283) (1.4,1.23108) (1.5,1.16786) (1.6,1.14105) (1.7,1.16347) (1.8,1.22412) (1.9,1.30468) (2.,1.39446)};
\addplot [color=BF-blue!75!BF-green,line width=0.7pt,smooth] coordinates {(-2.,1.67424) (-1.9,1.57835) (-1.8,1.48616) (-1.7,1.40303) (-1.6,1.34503) (-1.5,1.34356) (-1.4,1.3998) (-1.3,1.48229) (-1.2,1.57424) (-1.1,1.67002) (-1.,1.76759) (-0.9,1.86559) (-0.8,1.96328) (-0.7,2.06055) (-0.6,2.1572) (-0.5,2.25289) (-0.4,2.347) (-0.3,2.43824) (-0.2,2.52347) (-0.1,2.59374) (0.,2.62476) (0.1,2.59374) (0.2,2.52347) (0.3,2.43824) (0.4,2.347) (0.5,2.25289) (0.6,2.1572) (0.7,2.06055) (0.8,1.96328) (0.9,1.86559) (1.,1.76759) (1.1,1.67002) (1.2,1.57424) (1.3,1.48229) (1.4,1.3998) (1.5,1.34356) (1.6,1.34503) (1.7,1.40303) (1.8,1.48616) (1.9,1.57835) (2.,1.67424)};
\addplot [color=BF-blue!80!BF-green,line width=0.7pt,smooth] coordinates {(-2.,1.93389) (-1.9,1.83583) (-1.8,1.73938) (-1.7,1.64687) (-1.6,1.56665) (-1.5,1.53089) (-1.4,1.57449) (-1.3,1.65702) (-1.2,1.7502) (-1.1,1.84691) (-1.,1.94509) (-0.9,2.04361) (-0.8,2.14189) (-0.7,2.23984) (-0.6,2.33732) (-0.5,2.43405) (-0.4,2.52955) (-0.3,2.62276) (-0.2,2.71104) (-0.1,2.78609) (0.,2.82092) (0.1,2.78609) (0.2,2.71104) (0.3,2.62276) (0.4,2.52955) (0.5,2.43405) (0.6,2.33732) (0.7,2.23984) (0.8,2.14189) (0.9,2.04361) (1.,1.94509) (1.1,1.84691) (1.2,1.7502) (1.3,1.65702) (1.4,1.57449) (1.5,1.53089) (1.6,1.56665) (1.7,1.64687) (1.8,1.73938) (1.9,1.83583) (2.,1.93389)};
\addplot [color=BF-blue!85!BF-green,line width=0.7pt,smooth] coordinates {(-2.,2.17869) (-1.9,2.07973) (-1.8,1.98156) (-1.7,1.88526) (-1.6,1.79471) (-1.5,1.73041) (-1.4,1.75395) (-1.3,1.83579) (-1.2,1.92992) (-1.1,2.02731) (-1.,2.12593) (-0.9,2.22481) (-0.8,2.32351) (-0.7,2.42196) (-0.6,2.52003) (-0.5,2.61752) (-0.4,2.71403) (-0.3,2.80872) (-0.2,2.89937) (-0.1,2.97849) (0.,3.01708) (0.1,2.97849) (0.2,2.89937) (0.3,2.80872) (0.4,2.71403) (0.5,2.61752) (0.6,2.52003) (0.7,2.42196) (0.8,2.32351) (0.9,2.22481) (1.,2.12593) (1.1,2.02731) (1.2,1.92992) (1.3,1.83579) (1.4,1.75395) (1.5,1.73041) (1.6,1.79471) (1.7,1.88526) (1.8,1.98156) (1.9,2.07973) (2.,2.17869)};
\addplot [color=BF-blue!90!BF-green,line width=0.7pt,smooth] coordinates {(-2.,2.41288) (-1.9,2.31349) (-1.8,2.21453) (-1.7,2.11656) (-1.6,2.02146) (-1.5,1.9399) (-1.4,1.93743) (-1.3,2.01766) (-1.2,2.11254) (-1.1,2.21045) (-1.,2.30938) (-0.9,2.40853) (-0.8,2.50753) (-0.7,2.60633) (-0.6,2.70484) (-0.5,2.80288) (-0.4,2.90015) (-0.3,2.99596) (-0.2,3.08845) (-0.1,3.17099) (0.,3.21323) (0.1,3.17099) (0.2,3.08845) (0.3,2.99596) (0.4,2.90015) (0.5,2.80288) (0.6,2.70484) (0.7,2.60633) (0.8,2.50753) (0.9,2.40853) (1.,2.30938) (1.1,2.21045) (1.2,2.11254) (1.3,2.01766) (1.4,1.93743) (1.5,1.9399) (1.6,2.02146) (1.7,2.11656) (1.8,2.21453) (1.9,2.31349) (2.,2.41288)};
\addplot [color=BF-blue!95!BF-green,line width=0.7pt,smooth] coordinates {(-2.,2.63941) (-1.9,2.53979) (-1.8,2.44043) (-1.7,2.34164) (-1.6,2.2444) (-1.5,2.15413) (-1.4,2.12464) (-1.3,2.20195) (-1.2,2.29744) (-1.1,2.39575) (-1.,2.49491) (-0.9,2.59425) (-0.8,2.69347) (-0.7,2.79253) (-0.6,2.89136) (-0.5,2.98982) (-0.4,3.08765) (-0.3,3.1843) (-0.2,3.27824) (-0.1,3.36363) (0.,3.40939) (0.1,3.36363) (0.2,3.27824) (0.3,3.1843) (0.4,3.08765) (0.5,2.98982) (0.6,2.89136) (0.7,2.79253) (0.8,2.69347) (0.9,2.59425) (1.,2.49491) (1.1,2.39575) (1.2,2.29744) (1.3,2.20195) (1.4,2.12464) (1.5,2.15413) (1.6,2.2444) (1.7,2.34164) (1.8,2.44043) (1.9,2.53979) (2.,2.63941)};
\addplot [color=BF-blue!100!BF-green,line width=0.7pt,smooth] coordinates {(-2.,2.86029) (-1.9,2.76055) (-1.8,2.66097) (-1.7,2.56173) (-1.6,2.46341) (-1.5,2.36887) (-1.4,2.31573) (-1.3,2.38816) (-1.2,2.48414) (-1.1,2.58275) (-1.,2.68209) (-0.9,2.78157) (-0.8,2.88095) (-0.7,2.98021) (-0.6,3.07928) (-0.5,3.17805) (-0.4,3.27631) (-0.3,3.37361) (-0.2,3.46867) (-0.1,3.55644) (0.,3.60555) (0.1,3.55644) (0.2,3.46867) (0.3,3.37361) (0.4,3.27631) (0.5,3.17805) (0.6,3.07928) (0.7,2.98021) (0.8,2.88095) (0.9,2.78157) (1.,2.68209) (1.1,2.58275) (1.2,2.48414) (1.3,2.38816) (1.4,2.31573) (1.5,2.36887) (1.6,2.46341) (1.7,2.56173) (1.8,2.66097) (1.9,2.76055) (2.,2.86029)};
\addplot [color=gray,line width=0.4pt,dashed] coordinates {(-2,-0.4) (-2,3)};
\addplot [color=gray,line width=0.4pt,dashed] coordinates {(2,-0.4) (2,3)};
\node[label={270:{\scriptsize{-2}}}] at (axis cs:-2.2,0.04) {};
\node[label={270:{\scriptsize{0}}}] at (axis cs:-0.15,0.04) {};
\node[label={270:{\scriptsize{2}}}] at (axis cs:2.15,0.04) {};
\end{axis}
\end{tikzpicture}
\caption{The figure shows the graphs of the solutions of the Cauchy problem $(u(0),u'(0))=(R,0)$
associated with the equation in \eqref{eq-main-phi} with $a(t) = 1$ on $\mathopen{]}-1,1\mathclose{[}$, $a(t) = -10$ on $\mathopen{[}-2,-1\mathclose{]} \cup \mathopen{[}1,2\mathclose{]}$ and $g(u) = e^{u^{2}}-1$
(on the left) and $g(u)=u^{2}$ (on the right). 
The solutions satisfying Neumann boundary conditions (one on the left, and two on the right) are painted in red. 
In both cases, large solutions become more and more similar
to affine functions with slope $\pm 1$ (cf.~\cite[Section~4]{BoFe-PP} where this behavior is indeed proved for the parameter-dependent equation \eqref{eq-g} when $\lambda \to +\infty$). However, as proved in the verification of condition $(H_{2})$, for the super-exponential nonlinearity $g(u) = e^{u^{2}}-1$, large solutions are decreasing (with slope close to $-1$) on the right of the maximum point $t = 0$ and increasing (with slope close to $1$) 
on the left of the maximum point: therefore, their graphs are $\wedge$-shaped. On the contrary, for the power nonlinearity $g(u)=u^{2}$, the behavior of large solutions is the same on the positivity interval $\mathopen{[}-1,1\mathclose{]}$ (notice, indeed, that the proof of \eqref{eq-hat-t-2} only requires that $g(u) \to +\infty$ for $u \to +\infty$), but
further changes of monotonicity, with slopes suddenly passing from $-1$ to $1$ and viceversa, can arise in the negativity intervals.} 
\label{fig-03}
\end{figure}
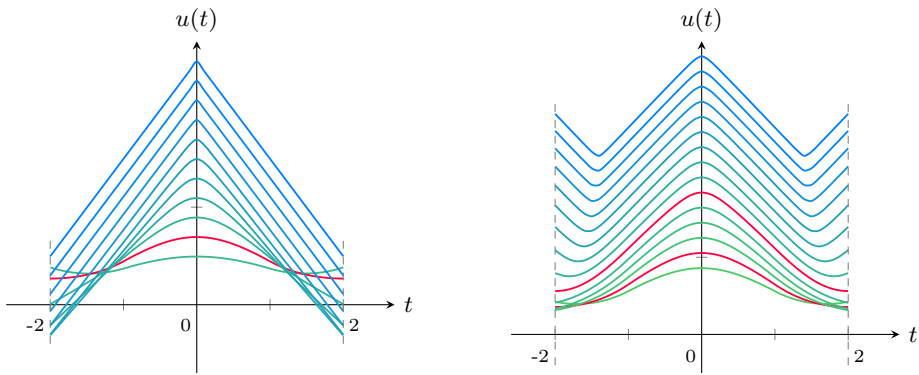

\begin{remark}\label{rem-3.2}
It is worth noticing that, arguing as in \cite[Lemma~3.4]{BoFe-PP}, it is possible to verify condition $(H_{1})$
also when condition $(g_{0})$ is replaced by the assumption that $g(u)$ is continuously differentiable in a right neighborhood of $u=0$ and $g'(0) = 0$.
Moreover, looking more deeply into the whole proof, one can observe that it is not necessary that $a(t)$ is essentially bounded in the whole interval $\mathopen{[}0,T\mathclose{]}$: it is sufficient that $a\in L^{1}(0,T)\cap L^{\infty}(\bigcup_{i=1}^{m} I^{-}_{i})$. Furthermore, concerning the monotonicity condition $(g_{\nearrow})$, we notice that it is assumed only for simplicity in the exposition, indeed Theorem~\ref{th-main-ex} holds true also when $(g_{\nearrow})$ is replaced by the following weaker assumption:
\begin{itemize}[leftmargin=34pt,labelsep=12pt,itemsep=6pt,topsep=5pt]
\item there exist $M\in\mathopen{]}0,1\mathclose{]}$ and $\check{R}>0$ such that $g(t)\geq M g(s)$ for all $s\geq \check{R}$ and $t\in\mathopen{[}s,s+\max_{i} \delta_{i}\mathclose{]}$.
\end{itemize}
If the above hypothesis is assumed instead of $(g_{\nearrow})$, minimal modifications in the verification of $(H_{2})$ lead to the same conclusion.
\hfill$\lhd$
\end{remark}

\bibliographystyle{elsart-num-sort}
\bibliography{BoFeZa-biblio}

\end{document}